\documentclass{amsart}
\usepackage{pgfplots}
\usepackage{amsmath}
\pgfplotsset{compat=1.18}
\usepackage{sansmath} % optional: math font consistent with sans serif text
\sansmath
\usepackage{graphicx}
\usepackage{slashed}
\usepackage{enumerate}
\usepackage{dsfont}
\usepackage{mathtools,cancel}
\usepackage[bottom]{footmisc}
\usepackage[scr=rsfso]{mathalfa}
\usepackage{tikz-cd}
\usepackage[hidelinks]{hyperref}
{

}
\usepackage{amsfonts}
\usepackage{psfrag}
\usepackage{color}
\usepackage{mathtools}
\usepackage{pgf,tikz}
\usepackage{mathrsfs}
\usetikzlibrary{arrows}
\usepackage{fancyhdr}
\usepackage{amssymb}
\usepackage{slashed}
\usepackage{enumitem}

\usepackage[T1]{fontenc}
\usepackage[utf8]{inputenc}
\newtheorem{remark}{Remark}

\newtheorem{proposition}{Proposition}[section]

\newtheorem{theorem}{Theorem}[section]
\newtheorem{corollary}{Corollary}[section]
\newtheorem{lemma}{Lemma}[section]

\newcommand{\be}{\begin{equation}}
\newcommand{\ee}{\end{equation}}

\newcommand{\bm}{\begin{align}*}
\newcommand{\enm}{\end{align}*}

\newcommand{\bespeq}{\begin{equation}\begin{split}}
\newcommand{\espeq}{\end{split}\end{equation}}

\def\S {S_{u,\underline{u}}}

\def\p{\psi}

\def\p{\psi}

\newcommand{\tr}{\mbox{tr}}

\newcommand\restri[2]{{% we make the whole thing an ordinary symbol
		\left.\kern-\nulldelimiterspace % automatically resize the bar with \right
		#1 % the function
		\right|_{#2} % this is the delimiter
}}

\definecolor{ffqqqq}{rgb}{1.,0.,0.}
\definecolor{uuuuuu}{rgb}{0.26666666666666666,0.26666666666666666,0.26666666666666666}

\makeatletter
\def\ps@pprintTitle{%
  \let\@oddhead\@empty
  \let\@evenhead\@empty
  \let\@oddfoot\@empty
  \let\@evenfoot\@oddfoot
}
\def\@author#1{\g@addto@macro\elsauthors{\normalsize%
    \def\baselinestretch{1}%
    \upshape\authorsep#1\unskip\textsuperscript{%
      \ifx\@fnmark\@empty\else\unskip\sep\@fnmark\let\sep=,\fi
      \ifx\@corref\@empty\else\unskip\sep\@corref\let\sep=,\fi
      }%
    \def\authorsep{\unskip,\space}%
    \global\let\@fnmark\@empty
    \global\let\@corref\@empty  %% Added
    \global\let\sep\@empty}%
    \@eadauthor={#1}
}
\makeatother
\begin{document}

\title{Volume Stability for Hyperbolic Manifolds and Applications to General Relativity}

\author{Puskar Mondal}\footnote{e-mail:pushkarmondal@gmail.com}%Beijing Institute of Mathematical Sciences and Applications, Yau Mathematical Sciences Center}
\author{Shing-Tung Yau} %\footnote{Beijing Institute of Mathematical Sciences and Applications, Yau Mathematical Sciences Center}

\maketitle

\begin{abstract}
We prove a sharp volume-stability theorem for closed hyperbolic three-manifolds. Let \((M,h)\) be closed hyperbolic with \(\operatorname{Ric}_h=-2h\), and let \(g_i\) be smooth metrics on \(M\) satisfying
$R(g_i)\geq -6,
\operatorname{Vol}_{g_i}(M)\longrightarrow \operatorname{Vol}_h(M)$.
After passing to a subsequence, there exist \(Z_i\subset M\), smooth domains \(K_i\subset M\), and diffeomorphisms
$\psi_i:K_i\longrightarrow M\setminus Z_i
$
such that
$\operatorname{Vol}_{g_i}(Z_i)\longrightarrow0,
\operatorname{Vol}_h(M\setminus K_i)\longrightarrow0,
$
and
$
\|\psi_i^*g_i-h\|_{C^0(K_i,h)}\longrightarrow0.
$
Thus near-equality in the sharp hyperbolic volume bound forces tensorial \(C^0\)-convergence to the hyperbolic metric outside regions of vanishing volume. As an application, we establish stability of the Fischer--Moncrief reduced Hamiltonian at the Lorentz-cone ground state: after CMC normalization, near-minimizing compact vacuum data in the hyperbolic topological class converge, modulo sets of vanishing volume, to the hyperbolic Lorentz-cone geometry in tensorial \(C^0\). This provides a rigorous volume-dominance formulation of the Fischer--Moncrief asymptotic picture.
\end{abstract}

\setcounter{tocdepth}{2}
{\hypersetup{linkcolor=black}
\small
\tableofcontents
}

\section{Introduction and the Main Result}
\noindent Recall the hyperbolic volume rigidity problem posed by Schoen. 
Let $(M^n,h)$ be a closed hyperbolic manifold, $n\geq 3$, normalized so that
\[
        \sec_h \equiv -1,
        \qquad
        R_h=-n(n-1).
\]
Let $g$ be another smooth Riemannian metric on $M$ satisfying
\[
        R_g \geq -n(n-1).
\]
Then
\[
        \operatorname{Vol}_g(M)\geq \operatorname{Vol}_h(M).
\]
Moreover, equality holds if and only if $g$ is hyperbolic, equivalently there exists a
diffeomorphism $\Phi:M\to M$ such that
\[
        g=\Phi^{*}h .
\]
Schoen’s conjecture is known to hold for 3-manifolds due to works of Hamilton on non-singular Ricci flow \cite{hamilton1999non} and Perelman on geometrization of 3-manifolds \cite{perelman2002entropy,perelman2003ricci} (also see
\cite{agol2007lower} for a generalization). For higher dimensions, Besson, Courtois and Gallot verified it for
metrics $C^{2}-$close to the canonical metric \cite{besson1991volume}. They also proved that the volume comparison
holds without assuming $g$ is close to $\overline{g}$, if one replace the assumption on scalar curvature by
Ricci curvature \cite{besson1995entropies}, which can be viewed as evidence that Schoen’s conjecture holds for
higher dimensions. In this article, we study a slight variant of the problem, namely quantification of the rigidity property in the regime of small data defined by the smallness of the volume deficit. More precisely we ask the following question: \textit{Let $(M,g_{i})$ be a sequence of smooth closed Riemannian 3 manifolds with scalar curvature bounds $R(g_{i})\geq -6$ and $\operatorname{Vol}(M,g_{i})-\operatorname{Vol}(M,h)\leq \delta_{i}\to 0$. Then in which topology can there be a convergence of $(M,g_{i})$ to $(M,h)?$} This question is reminiscent of the stability of the positive mass theorem in general relativity \cite{dong2025stability} and in particular, can be considered to be an equivalent in the case of spatially compact spacetimes where a certain weak Lyapunov functional (re-scaled volume of the Cauchy slice) plays the role of ``\textit{mass}". The main point is that the volume is too weak to control any type of strong convergence. As it turns out, outside some bad sets of negligible volume, there is a $C^{0}$ convergence to the hyperbolic manifold (this can be improved under certain additional assumptions). We prove the following theorem by means of the normalized Ricci flow.  

\begin{theorem}[Main Theorem]
\label{main1}
Let \((M,h)\) be a closed hyperbolic three-manifold normalized by
\[
        \operatorname{Ric}_{h}=-2h,
        \qquad
        R(h)=-6 .
\]
Let \(g_i\) be smooth metrics on \(M\) satisfying
\[
        R(g_i)\geq -6,
        \qquad
        0\leq
        \operatorname{Vol}_{g_i}(M)-\operatorname{Vol}_{h}(M)
        \leq \delta_i,
        \qquad
        \delta_i\to0 .
\]
Then, after passing to a subsequence, there exist 
$Z_i\subset M,~
        G_i:=M\setminus Z_i$,
smooth compact domains
$K_i\subset M$,
and maps
 $\Phi_i:G_i\longrightarrow K_i$
which are diffeomorphisms onto their images, such that
$\operatorname{Vol}_{g_i}(Z_i)\to0,
$
and
$ \operatorname{Vol}_{h}(M\setminus K_i)\to0$.
Moreover,
\[
        \sup_{x\in G_i}
        \left|
        g_i-(\Phi_i)^*(h|_{K_i})
        \right|_{(\Phi_i)^*h}(x)
        \to0,
\]
that is, there exist \(\varepsilon_i\to0\) such that
\[
        (1-\varepsilon_i)\Phi_i^*h
        \leq
        g_i
        \leq
        (1+\varepsilon_i)\Phi_i^*h
        \qquad
        \text{on }G_i 
\]
as quadratic forms.
\end{theorem}

\begin{remark}
Note we do not have an estimate on the $\mathcal{H}^{2}$ estimate of the boundary $\partial Z_{i}$ of the bad set $Z_{i}$, and the $C^{0}$ convergence is a requirement for the technique used (see remark 2). 
\end{remark}

\subsection{Main ideas of the proof and the principal difficulties}
\medskip
\noindent
We now describe the main mechanism of the proof. The fundamental idea is to use the Ricci flow with Surgery introduced by Hamilton \cite{hamilton1982three} and later by Perelman \cite{perelman2002entropy,perelman2003ricci} (note the expositions \cite{cao2006complete,kleiner2008notes}). The main result in this context is the global well-posedness result of Ricci flow with surgery. We explain precisely where long-time Ricci-flow convergence enters. We first fix the Ricci-flow continuation used in the argument.  For each
initial metric \(g_i\), let \(\mathcal G_i(t)\), \(t\geq0\), be a
three-dimensional normalized Ricci flow with surgery in the sense of
Hamilton-Perelman, starting from \(g_i\).  Thus, on every smooth time interval,
\[
        \partial_t \mathcal G_i
        =
        -2(\operatorname{Ric}_{\mathcal G_i}+2\mathcal G_i),
\]
and the flow is continued through surgery times by the standard
precise surgery procedure.  We use the following properties of
this flow:

\[
        R(\mathcal G_i(t))+6\geq0
\]
is preserved on every post-surgery time slice,

\[
        \frac{d}{dt}\operatorname{Vol}(\mathcal G_i(t))
        =
        -\int_{M} (R(\mathcal G_i(t))+6)\,d\mu_{\mathcal G_i(t)}
\]
on smooth time intervals, and therefore volume is nonincreasing across surgery times. The
hyperbolic volume comparison theorem together with the scalar curvature bound gives only the lower barrier
\[
        \operatorname{Vol}_{g_i(t)}(M)\geq V_h .
\]
It does not by itself imply that \(g_i(t)\) is close to \(h\) in any strong sense since volume is too weak to control any strong convergence.  Ricci flow is
used to turn near equality in volume into geometric convergence. To understand how Ricci flow can be utilized in this context, we first observe the \textbf{rigid} mechanism.  If the volume equality holds with scalar curvature bound, then
\[
        \operatorname{Vol}_{g(0)}(M)=V_h,
        \qquad
        R(g(0))\geq -6,
\]
then along the normalized flow
\[
        \partial_t g=-2(\operatorname{Ric}+2g),~g(t=0)=g(0)
\]
we have
\[
        V_h\leq \operatorname{Vol}_{g(t)}(M)\leq \operatorname{Vol}_{g(0)}(M)=V_h,
\]
since the scalar curvature lower bound is preserved by the chosen normalized Ricci flow and by Schoen's volume comparison result.
Hence \(\operatorname{Vol}_{g(t)}(M)\equiv V_h\).  Since for $t\in [0,T]$ in every regular time interval
\[
        \frac{d}{dt}\operatorname{Vol}_{g(t)}(M)
        =
        -\int_M (R(g(t))+6)\,d\mu_{g(t)},
\]
since $R(g(t))+6\geq 0$ along the Ricci flow, we get
\[
        R(g(t))+6\equiv0.
\]
But the following sourced heat type equations is verified by $R(g(t))+6$
\[
        (\partial_t-\Delta+4)(R(g(t))+6)
        =
        2|\operatorname{Ric}+2g|^2.
\]
Therefore
\[
        \operatorname{Ric}_{g(t)}+2g(t)\equiv0.
\]
Thus \(g(t)\) is hyperbolic. Since \(M\) is closed hyperbolic, Mostow
rigidity identifies \(g(t)\) with \(h\), up to diffeomorphism.

\noindent In the stability case, let \(g_i(t)\) denote the distinguished
hyperbolic component of the normalized Ricci flow with sufficiently
precise surgery. On every regular time interval,
\[
Q_i:=R(g_i(t))+6
\]
satisfies
\[
(\partial_t-\Delta_{g_i(t)}+4)Q_i
=
2\bigl|\operatorname{Ric}_{g_i(t)}+2g_i(t)\bigr|^2
\geq0.
\]
The standard surgery construction preserves \(Q_i\geq0\) across surgery
times. Hence, by hyperbolic volume comparison,
\[
V_h\leq \operatorname{Vol}_{g_i(t)}(M)\leq
\operatorname{Vol}_{g_i(0)}(M)\leq V_h+\delta_i.
\]
Moreover, if \(\Delta_i^{\mathrm{surg}}(s)\geq0\) denotes the volume
lost at a surgery time \(s\), then
\[
\operatorname{Vol}_{g_i(0)}(M)-\operatorname{Vol}_{g_i(T)}(M)
=
\int_0^T\int_M Q_i\,d\mu_{g_i(t)}\,dt
+
\sum_{0<s<T}\Delta_i^{\mathrm{surg}}(s).
\]
Consequently,
\begin{equation}
\label{eq:Qspacetime}
\int_0^\infty\int_M
\bigl(R(g_i(t))+6\bigr)\,d\mu_{g_i(t)}\,dt
\leq \delta_i,
\end{equation}
and the total surgery volume loss is bounded by the same deficit.

\noindent Similarly, let
\[
\overline\lambda[g]
:=
\lambda[g]\operatorname{Vol}_{g}(M)^{2/3},
\qquad
\mathcal S_i
:=
\operatorname{Ric}_{g_i}
+\nabla^2 f_i-\frac{\lambda[g_i]}{3}g_i,
\]
where \(f_i\) is the normalized minimizer for Perelman's
\(\lambda\)-functional. On every regular interval,
\[
\frac{d}{dt}\overline\lambda[g_i(t)]
\geq
2\operatorname{Vol}_{g_i(t)}(M)^{2/3}
\int_M|\mathcal S_i(t)|^2e^{-f_i(t)}\,d\mu_{g_i(t)}.
\]
Choosing the surgery parameters so that the total entropy loss at
surgery is absorbed into \(C\delta_i\), and using
\[
0\leq \overline\lambda[h]-\overline\lambda[g_i(0)]
\leq C_h\delta_i,
\]
we obtain
\begin{equation}
\label{eq:spacetime}
\int_0^\infty\int_M
|\mathcal S_i(t)|^2e^{-f_i(t)}\,d\mu_{g_i(t)}\,dt
\leq C_h\delta_i.
\end{equation}
Here the spacetime integrals are taken over the regular portions of the
surgical flow.
Thus, after choosing a sequence of large times \(A_i\to\infty\), there exist regular times
$ t_i\in[A_i,2A_i]$
such that
\[
        \int_M
        \bigl(R(g_i(t_i))+6\bigr)\,d\mu_{g_i(t_i)}
        \to0,
\]
and
\[
        \int_M
        |\mathcal S_i(t_i)|^2 e^{-f_i(t_i)}\,d\mu_{g_i(t_i)}
        \to0.
\]
Set
\[
        \widehat g_i:=g_i(t_i).
\]
Now the long-time theorem is used.  Perelman's long-time geometrization says
that the normalized metrics \(\widehat g_i\) have a thick-thin
decomposition.  On the thick part, after passing to a subsequence, one has
smooth pointed Cheeger--Gromov convergence to finite-volume hyperbolic
pieces.  The thin part is collapsed with sectional curvature bounded from
below.

\noindent Because the underlying manifold is the fixed closed hyperbolic manifold
\(M\), its geometrization has a single hyperbolic piece, namely \(M\)
itself.  Hence the thick hyperbolic limit has volume
\[
        V_h=\operatorname{Vol}_h(M).
\]
By Mostow rigidity, this limiting hyperbolic metric is isometric to \(h\).
Therefore, for every compact exhaustion
\[
        K^j\Subset (M,h),
        \qquad
        K^j\nearrow M,
\]
there exist diffeomorphisms onto their images
\[
        \psi_i^j:K^j\to U_i^j\subset M
\]
such that
\[
        (\psi_i^j)^*\widehat g_i\to h
        \quad\text{smoothly on }K^j.
\]

\noindent The bad-volume estimate is then purely bookkeeping.  Choose \(j=j(i)\to
\infty\) slowly and set
\[
        K_i^t:=K^{j(i)},
        \qquad
        G_i^t:=U_i^{j(i)},
        \qquad
        Z_i^t:=M\setminus G_i^t .
\]
Then
\[
        \|(\psi_i^t)^*\widehat g_i-h\|_{C^0(K_i^t,h)}\to0.
\]
Moreover,
\[
\begin{aligned}
        \operatorname{Vol}_{\widehat g_i}(Z_i^t)
        &=
        |\operatorname{Vol}_{\widehat g_i}(M)
        -
        \operatorname{Vol}_{\widehat g_i}(G_i^t)|         =|\operatorname{Vol}_{\widehat g_i}(M)-\operatorname{Vol}_h(M)+\operatorname{Vol}_h(K_i^t)-\operatorname{Vol}_{\widehat g_i}(G_i^t)+\operatorname{Vol}_h(M)-\operatorname{Vol}_h(K_i^t)|\\
        &\leq
        \left|
        \operatorname{Vol}_{\widehat g_i}(M)-\operatorname{Vol}_h(M)
        \right|                                          +
        \left|
        \operatorname{Vol}_{\widehat g_i}(G_i^t)
        -
        \operatorname{Vol}_h(K_i^t)
        \right|                                          +
        \operatorname{Vol}_h(M\setminus K_i^t).
\end{aligned}
\]
The first term tends to zero by the volume comparison and monotonicity.  The
second tends to zero by smooth convergence on \(K_i^t\).  The third tends to
zero by the exhaustion.  Hence
\[
        \operatorname{Vol}_{\widehat g_i}(Z_i^t)\to0.
\]

\noindent The remaining difficulty is to transfer this information back to the initial
slice.  This is the genuinely delicate part of the argument.  The normalized
Ricci flow deforms the metric by
\[
        \partial_t g_i(t)=-2E_i(t),
\]
so metric comparison along a worldline requires control of the accumulated
quantity
\[
        \mathcal A_i(x)
        :=
        \int_0^{t_i}
        |E_i(s)|_{g_i(s)}(\Psi_i(s,x))\,ds,
\]
where \(\Psi_i(s,x)\) denotes the surviving worldline starting from \(x\) at time \(0\).
The spacetime estimate \eqref{eq:spacetime} alone gives only
\(L^2_tL^2_x\)-control of \(E_i\), which is insufficient when the comparison
time \(t_i\) is large.  Indeed, a naive Cauchy--Schwarz estimate produces a
factor of \(t_i\), and therefore does not rule out large accumulated metric
distortion.

\noindent The correct replacement is an \(L^1_tL^2_x\)-path-length estimate:
\begin{equation}
\label{eq:L1L2-pathlength}
        \mathcal L_i
        :=
        \int_0^{t_i}
        \|E_i(t)\|_{L^2(\mathcal G_i(t),g_i(t))}\,dt
        \longrightarrow0
\end{equation}
on the relevant good spacetime tube \(\mathcal G_i\).  This estimate is proved
using the gradient-flow structure of the normalized Ricci flow, the entropy
monotonicity, and a Lojasiewicz inequality near the hyperbolic metric. 

\noindent We also need to control the change of volume measure along the same worldlines.
Let \(\mathcal G_i^0\subset M\) be the initial trace of a regular spacetime tube
saturated by worldlines, and write
$\Psi_i(t):\mathcal G_i^0\to\mathcal G_i(t).$
If
\[
        d\mu_{g_i(t)}(\Psi_i(t,x))
        =
        J_i(t,x)\,d\mu_{g_i(0)}(x),
\]
then
\[
        \frac{d}{dt}\log J_i(t,x)
        =
        -Q_i(t,\Psi_i(t,x)).
\]
Thus
\[
        J_i(t,x)
        =
        \exp\left(
        -\int_0^t Q_i(s,\Psi_i(s,x))\,ds
        \right).
\]
For \(\kappa_i\to0\) with \(\delta_i/\kappa_i\to0\), define
\[
        \mathcal B_i^Q(0;\kappa_i)
        :=
        \left\{
        x\in\mathcal G_i^0:
        \int_0^{t_i} Q_i(s,\Psi_i(s,x))\,ds>\kappa_i
        \right\}.
\]
Using \eqref{eq:Qspacetime}, one obtains
\[
        \operatorname{Vol}_{g_i(0)}
        \bigl(\mathcal B_i^Q(0;\kappa_i)\bigr)
        \leq
        \frac{\delta_i}{1-e^{-\kappa_i}}
        \to0 .
\]
On the complement, the Jacobian satisfies
\begin{equation}
\label{eq:jacobian-good}
        e^{-\kappa_i}\leq J_i(t,x)\leq1 .
\end{equation}

\noindent Combining \eqref{eq:L1L2-pathlength} with \eqref{eq:jacobian-good} gives
\[
\begin{aligned}
        \|\mathcal A_i\|_{L^2(\mathcal G_i^{0,Q},g_i(0))}
        &\leq
        e^{\kappa_i/2}
        \int_0^{t_i}
        \|E_i(t)\|_{L^2(\mathcal G_i(t),g_i(t))}\,dt  =
        e^{\kappa_i/2}\mathcal L_i ,
\end{aligned}
\]
where
\[
        \mathcal G_i^{0,Q}
        :=
        \mathcal G_i^0\setminus \mathcal B_i^Q(0;\kappa_i).
\]
Now choose \(\eta_i\to0\) with
$\frac{\mathcal L_i^2}{\eta_i^2}\to0$,
and remove the Einstein-distortion bad set
\[
        \mathcal B_i^E(0;\eta_i)
        :=
        \{x\in\mathcal G_i^{0,Q}:\mathcal A_i(x)>\eta_i\}.
\]
Chebyshev's inequality gives
\[
        \operatorname{Vol}_{g_i(0)}
        \bigl(\mathcal B_i^E(0;\eta_i)\bigr)
        \leq
        e^{\kappa_i}\frac{\mathcal L_i^2}{\eta_i^2}
        \to0 .
\]
Therefore, on the final initial good set
\[
        \mathcal G_i^{0,*}
        :=
        \mathcal G_i^{0,Q}\setminus \mathcal B_i^E(0;\eta_i),
\]
the accumulated metric distortion satisfies
\[
        \mathcal A_i(x)\leq\eta_i .
\]
Integrating the evolution equation for the metric along worldlines gives the
bilipschitz comparison
\begin{equation}
\label{eq:metric-comparison-initial-final}
        e^{-2\eta_i}g_i(0)
        \leq
        \Psi_i(t_i)^*g_i(t_i)
        \leq
        e^{2\eta_i}g_i(0)
\end{equation}
as quadratic forms on \(\mathcal G_i^{0,*}\).

\noindent This comparison is the key point of the proof.  It allows the convergence of
the good time slices produced by the Ricci flow to be pulled back to the
original sequence \((M,g_i)\), after removing subsets whose \(g_i\)-volume tends
to zero.  In this way the scalar-curvature lower bound and the sharp volume
deficit yield $C^{0}$ convergence to the hyperbolic metric on the initial manifolds,
up to negligible exceptional sets.
\begin{remark}
Note that while pulling back from the intermediate good slice $(t=t_{i},t_{i}\to\infty)$, we can not control the higher order norm of the deformation tensor $\partial_{t}g$ even on the good part since such estimates are simply not available to us. Therefore, we can only obtain a $C^{0}$ type convergence result on the volume dominating good sets.     
\end{remark}

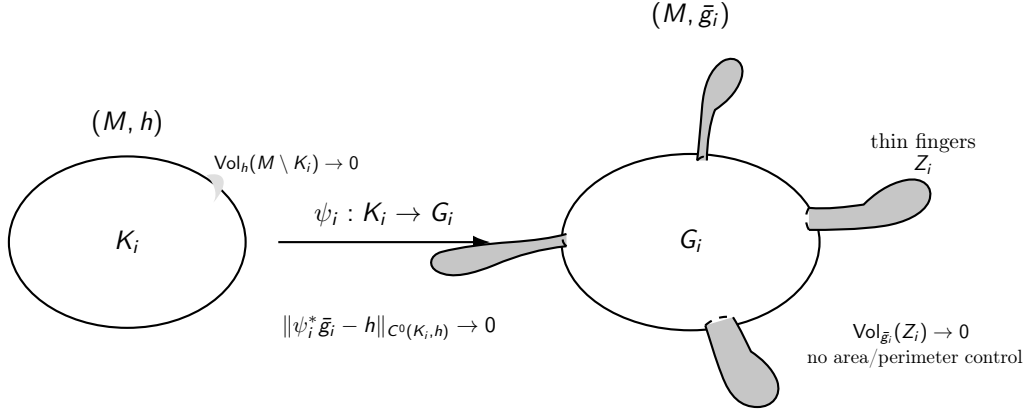
\begin{figure}[t]
\centering
\begin{tikzpicture}[scale=1.08, >=Latex]

%---------------- Left: reference hyperbolic manifold ----------------%
\draw[thick] (-4.8,0) ellipse (1.45 and 1.05);
\node at (-4.8,1.45) {\((M,h)\)};
\node at (-4.8,0) {\(K_i\)};

% tiny omitted set on reference side
\fill[gray!25]
(-3.75,0.48)
.. controls (-3.55,0.70) and (-3.58,0.88) .. (-3.83,0.82)
.. controls (-3.70,0.68) and (-3.70,0.55) .. (-3.75,0.48)
-- cycle;

\node[scale=0.72] at (-2.85,0.95) {\(\operatorname{Vol}_h(M\setminus K_i)\to0\)};

% arrow
\draw[->, thick] (-2.95,0) -- (-0.35,0);
\node[above] at (-1.65,0.08) {\(\psi_i:K_i\to G_i\)};

\node[align=center, scale=0.84] at (-1.55,-1.05)
{
\(\displaystyle \|\psi_i^*\bar g_i-h\|_{C^0(K_i,h)}\to0\)
};

%---------------- Right: manifold with thin negligible fingers ----------------%
% main body
\draw[thick] (2.1,0) ellipse (1.58 and 1.08);
\node at (2.1,0.02) {\(G_i\)};

% moved upward so it is not hidden by the top finger
\node at (2.1,2.78) {\((M,\bar g_i)\)};

%---------- top thin finger ----------%
\fill[gray!45]
(2.18,1.00)
.. controls (2.20,1.16) and (2.22,1.40) .. (2.24,1.64)
.. controls (2.26,1.92) and (2.34,2.12) .. (2.50,2.22)
.. controls (2.66,2.30) and (2.78,2.18) .. (2.72,2.00)
.. controls (2.66,1.82) and (2.52,1.70) .. (2.40,1.62)
.. controls (2.34,1.48) and (2.30,1.22) .. (2.30,1.00)
-- cycle;
\draw[thick]
(2.18,1.00)
.. controls (2.20,1.16) and (2.22,1.40) .. (2.24,1.64)
.. controls (2.26,1.92) and (2.34,2.12) .. (2.50,2.22)
.. controls (2.66,2.30) and (2.78,2.18) .. (2.72,2.00)
.. controls (2.66,1.82) and (2.52,1.70) .. (2.40,1.62)
.. controls (2.34,1.48) and (2.30,1.22) .. (2.30,1.00);
\draw[dashed, thick] (2.18,1.00) .. controls (2.24,1.06) and (2.28,1.06) .. (2.30,1.00);

%---------- right ultra-thin finger ----------%
\fill[gray!45]
(3.55,0.15)
.. controls (3.80,0.15) and (4.15,0.15) .. (4.48,0.20)
.. controls (4.72,0.24) and (4.95,0.34) .. (5.02,0.50)
.. controls (5.09,0.68) and (4.94,0.80) .. (4.74,0.76)
.. controls (4.52,0.72) and (4.34,0.58) .. (4.22,0.46)
.. controls (4.00,0.43) and (3.78,0.40) .. (3.55,0.42)
-- cycle;
\draw[thick]
(3.55,0.15)
.. controls (3.80,0.15) and (4.15,0.15) .. (4.48,0.20)
.. controls (4.72,0.24) and (4.95,0.34) .. (5.02,0.50)
.. controls (5.09,0.68) and (4.94,0.80) .. (4.74,0.76)
.. controls (4.52,0.72) and (4.34,0.58) .. (4.22,0.46)
.. controls (4.00,0.43) and (3.78,0.40) .. (3.55,0.42);
\draw[dashed, thick] (3.55,0.15) .. controls (3.50,0.24) and (3.50,0.34) .. (3.55,0.42);

%---------- bottom thin finger with tiny bulb ----------%
\fill[gray!45]
(2.28,-1.02)
.. controls (2.34,-1.20) and (2.40,-1.42) .. (2.48,-1.62)
.. controls (2.58,-1.90) and (2.72,-2.02) .. (2.92,-2.02)
.. controls (3.10,-2.02) and (3.22,-1.88) .. (3.16,-1.70)
.. controls (3.10,-1.54) and (2.96,-1.42) .. (2.84,-1.34)
.. controls (2.78,-1.22) and (2.72,-1.08) .. (2.62,-0.92)
-- cycle;
\draw[thick]
(2.28,-1.02)
.. controls (2.34,-1.20) and (2.40,-1.42) .. (2.48,-1.62)
.. controls (2.58,-1.90) and (2.72,-2.02) .. (2.92,-2.02)
.. controls (3.10,-2.02) and (3.22,-1.88) .. (3.16,-1.70)
.. controls (3.10,-1.54) and (2.96,-1.42) .. (2.84,-1.34)
.. controls (2.78,-1.22) and (2.72,-1.08) .. (2.62,-0.92);
\draw[dashed, thick] (2.28,-1.02) .. controls (2.38,-0.93) and (2.50,-0.90) .. (2.62,-0.92);

%---------- left thin spike ----------%
\fill[gray!45]
(0.58,-0.05)
.. controls (0.28,-0.08) and (-0.02,-0.16) .. (-0.30,-0.28)
.. controls (-0.55,-0.38) and (-0.82,-0.46) .. (-1.00,-0.36)
.. controls (-1.12,-0.30) and (-1.10,-0.16) .. (-0.95,-0.12)
.. controls (-0.72,-0.06) and (-0.40,-0.02) .. (-0.10,0.00)
.. controls (0.18,0.02) and (0.40,0.05) .. (0.58,0.10)
-- cycle;
\draw[thick]
(0.58,-0.05)
.. controls (0.28,-0.08) and (-0.02,-0.16) .. (-0.30,-0.28)
.. controls (-0.55,-0.38) and (-0.82,-0.46) .. (-1.00,-0.36)
.. controls (-1.12,-0.30) and (-1.10,-0.16) .. (-0.95,-0.12)
.. controls (-0.72,-0.06) and (-0.40,-0.02) .. (-0.10,0.00)
.. controls (0.18,0.02) and (0.40,0.05) .. (0.58,0.10);
\draw[dashed, thick] (0.58,-0.05) .. controls (0.52,0.00) and (0.52,0.05) .. (0.58,0.10);

% labels
\node[scale=0.8] at (4.95,1.22) {thin fingers};
\node[scale=0.8] at (4.95,0.95) {\(Z_i\)};
\node[scale=0.76] at (4.8,-1.12) {\(\operatorname{Vol}_{\bar g_i}(Z_i)\to0\)};
\node[scale=0.72] at (4.85,-1.42) {no area/perimeter control};

\end{tikzpicture}
\caption{The schematics of the main theorem \ref{main1}.  The exceptional set \(Z_i\) consists of thin
fingers and tiny bulbs, indicating regions with negligible
\(\bar g_i\)-volume.  On the complement \(G_i=M\setminus Z_i\), the maps
\(\psi_i:K_i\to G_i\) identify \(\bar g_i\) with \(h\) in tensorial
\(C^0\)-topology. 
}
\label{fig:thin-fingers-volume-negligible}
\end{figure}

\subsection{Applications to general relativity}

We finally explain the relevance of the preceding volume stability theorem to the spatially compact vacuum Einstein equations.  The analogy is with the
stability problem for the positive mass theorem in the asymptotically flat
setting \cite{dong2025stability}
: in both cases, a sharp scalar invariant controls the geometry only
after one removes a negligible exceptional region (Also note entropy stability result of \cite{song2025entropy}). In the present setting the
role of the ADM mass is played by the Fischer-Moncrief reduced Hamiltonian,
and the model geometry is the Lorentz cone over a closed hyperbolic
three-manifold.

\noindent Let \(M\) be a closed oriented three-manifold of negative Yamabe type, and let
\((\mathbf M,\mathbf g)\) be a globally hyperbolic vacuum spacetime with Cauchy
slices diffeomorphic to \(M\).  On a constant mean curvature hypersurface, the
induced initial data \((g,k)\) satisfy the vacuum constraint equations
\begin{align}
\label{eq:HC}
        R(g)-|k|_{g}^{2}+(\tr_g k)^2 &=0,\\
\label{eq:MC}
        \operatorname{div}_{g}\bigl(k-(\tr_g k)g\bigr)&=0.
\end{align}
Writing
\[
        k=\Sigma+\frac{\tau}{3}g,
        \qquad
        \tau:=\tr_g k,
\]
the momentum constraint implies, in CMC gauge, that
\[
        \operatorname{tr}_g\Sigma=0,
        \qquad
        \operatorname{div}_g\Sigma=0.
\]
Thus \(\Sigma\) is a transverse-traceless tensor with respect to \(g\).

\noindent Fischer and Moncrief introduced the reduced Hamiltonian
\begin{equation}
\label{eq:FM-Hamiltonian}
        \mathcal H(g,k):=(-\tau)^3\operatorname{Vol}_g(M),
        \qquad \tau=\tr_g k<0,
\end{equation}
which is a weak Lyapunov functional for the vacuum Einstein evolution in CMC
gauge.  More precisely, along every non-stationary CMC Einstein flow,
\(\mathcal H\) is monotone non-increasing, and equality occurs only on the
Lorentz-cone solutions.  These are the vacuum spacetimes foliated by homothetic
compact hyperbolic slices.

\noindent We recall the Fischer--Moncrief monotonicity formula for the reduced
Hamiltonian. Let
\[
        \mathbf g
        =
        -N^2dt^2
        +
        g_{ab}(dx^a+X^a dt)(dx^b+X^b dt)
\]
be a vacuum spacetime foliated by compact CMC hypersurfaces
\((M,g(t),k(t))\), and set as usual
\[
        \tau(t):=\operatorname{tr}_{g(t)}k(t)<0,
        \qquad
        k=\Sigma+\frac{\tau}{3}g.
\]
The lapse equation obtained by tracing the evolution equation for
\(k\) is
\begin{equation}
\label{eq:CMC-lapse}
        -\Delta_g N+|k|_g^2N=\dot\tau .
\end{equation}
Since \(M\) is closed, integration of \eqref{eq:CMC-lapse} gives
\begin{equation}
\label{eq:integrated-lapse}
        \dot\tau\,\operatorname{Vol}_g(M)
        =
        \int_M N|k|_g^2\,d\mu_g
        =
        \int_M N
        \left(
        |\Sigma|_g^2+\frac{\tau^2}{3}
        \right)d\mu_g .
\end{equation}

\noindent On the other hand, from
\[
        \partial_t g=-2Nk+\mathcal L_Xg
\]
one obtains
\[
        \frac{d}{dt}\operatorname{Vol}_g(M)
        =
        -\tau\int_M N\,d\mu_g,
\]
since the divergence term arising from the shift integrates to zero.
Therefore, for
\[
        \mathcal H(g,k)
        :=
        (-\tau)^3\operatorname{Vol}_g(M),
\]
we compute
\begin{align}
        \frac{d}{dt}\mathcal H(g,k)
        &=
        -3\tau^2\dot\tau\,\operatorname{Vol}_g(M)
        +
        \tau^4\int_M N\,d\mu_g \notag\\
        &=
        -3\tau^2
        \int_M N
        \left(
        |\Sigma|_g^2+\frac{\tau^2}{3}
        \right)d\mu_g
        +
        \tau^4\int_M N\,d\mu_g \notag\\
        &=
        -3\tau^2
        \int_M N|\Sigma|_g^2\,d\mu_g
        \leq0.
\label{eq:reduced-Hamiltonian-monotonicity}
\end{align}
Thus \(\mathcal H\) is nonincreasing along every expanding CMC vacuum
evolution and is strictly decreasing at every time for which
\(\Sigma\not\equiv0\).

\noindent Moreover, if \(\mathcal H\) is constant on a time interval, then
\(\Sigma\equiv0\) there. Equation \eqref{eq:CMC-lapse} then gives
\[
        N=\frac{3\dot\tau}{\tau^2},
\]
and the Einstein evolution equations imply
\[
        k=\frac{\tau}{3}g,
        \qquad
        \operatorname{Ric}_g=-\frac{2\tau^2}{9}g.
\]
Consequently,
\[
        h:=\frac{\tau^2}{9}g
        \qquad\text{satisfies}\qquad
        \operatorname{Ric}_h=-2h.
\]
After introducing \(T=-3/\tau\), the corresponding spacetime metric is,
up to diffeomorphism,
\[
        \mathbf g=-dT^2+T^2h,
\]
namely the Lorentz cone over \((M,h)\). Hence the reduced Hamiltonian is
constant along a solution curve precisely for the Lorentz-cone
solutions.

\noindent Finally, under the CMC normalization
\[
        \bar g=\frac{\tau^2}{9}g,
\]
one has
\[
        \mathcal H(g,k)
        =
        (-\tau)^3\operatorname{Vol}_g(M)
        =
        27\operatorname{Vol}_{\bar g}(M).
\]
Thus the Fischer--Moncrief monotonicity is equivalently the
monotonicity of the volume of the normalized spatial metric (conformal volume):
\[
        \frac{d}{dt}\operatorname{Vol}_{\bar g}(M)
        =
        -\frac{\tau^2}{9}
        \int_M N|\Sigma|_g^2\,d\mu_g
        \leq0.
\]

\noindent After this standard CMC normalization, the spatial metric and the transverse-
traceless part of \(k\) may be rescaled so that the Hamiltonian constraint takes
the normalized form
\begin{equation}
\label{eq:normalized-HC}
        R(g)+6=|\Sigma|_g^2,
        \qquad
        \operatorname{div}_g\Sigma=0.
\end{equation}
In particular,
\[
        R(g)\geq -6,
\]
with equality precisely when \(\Sigma=0\).  Under the same normalization, the
reduced Hamiltonian is, up to a fixed dimensional constant depending only on
the convention for \(\tau\), the volume of the normalized spatial metric.  Thus
near-minimizers of the Fischer--Moncrief Hamiltonian \cite{fischer2000reduced,fischer2001reduced} are exactly near-minimizers
of the hyperbolic volume among metrics satisfying the sharp lower scalar
curvature bound.

\noindent For a closed hyperbolic three-manifold \((M,h)\), normalized by
\[
        R(h)=-6,
\]
the Lorentz cone over \((M,h)\) is the expected ground state.  Equivalently, the
infimum of the Fischer-Moncrief Hamiltonian \cite{fischer2000reduced,fischer2001reduced} is determined by the Yamabe
invariant of \(M\).  With the usual normalization, this infimum may be written
as
\[
        \mathcal H_{\inf}
        =
        \left(-\frac{3}{2}\sigma(M)\right)^{3/2},
\]
up to the corresponding fixed convention in the definition of
\(\mathcal H\).  The essential point is that the minimizing configurations are
exactly the hyperbolic CMC data
\[
        \Sigma=0,
        \qquad
        \text{Ric}(g)=-2g.
\]

\noindent This leads to the following stability question.  Suppose a sequence of vacuum
CMC data \((g_i,k_i)\) has reduced Hamiltonian converging to the value of the
Lorentz-cone model.  Since the Hamiltonian is only a scalar quantity, one cannot
expect smooth convergence of \((g_i,k_i)\).  This is exactly analogous to the
positive mass stability problem: a sharp scalar invariant can concentrate its
defect on small regions and therefore cannot, by itself, control the geometry
in a strong topology.  The correct conclusion is a $C^{0}$
type convergence away from exceptional sets of vanishing volume, together with decay of the transverse-traceless part of the
second fundamental form.

\noindent The corresponding statement is the following.

\begin{theorem}[Reduced-Hamiltonian stability]
\label{main2}
Let \((M,h)\) be a closed hyperbolic three-manifold with
$ \operatorname{Ric}_h=-2h .$
Let \((M,g_i,k_i)\) be smooth vacuum CMC initial data with
\[
        \tau_i:=\operatorname{tr}_{g_i}k_i<0,
        \qquad
        k_i=\Sigma_i+\frac{\tau_i}{3}g_i,
\]
and set
\[
        \bar g_i:=\frac{\tau_i^2}{9}g_i,
        \qquad
        \bar\Sigma_i:=\frac{|\tau_i|}{3}\Sigma_i .
\]
Then the Hamiltonian constraint reduces to
$ R(\bar g_i)+6=|\bar\Sigma_i|_{\bar g_i}^{2}\geq 0$.
Assume the reduced Hamiltonian closeness of the physical data and that corresponds to a slice of the Lorentz cone spacetime
\[
        \mathcal H(g_i,k_i)\longrightarrow \mathcal{H}(h,k)=27\operatorname{Vol}_h(M).
\]
Then, after passing to a subsequence, there exist
$ Z_i\subset M,~ G_i:=M\setminus Z_i,$
compact domains \(K_i\subset M\), and diffeomorphisms
$\psi_i:K_i\longrightarrow G_i$
such that
\[
        \operatorname{Vol}_{\bar g_i}(Z_i)\to0,
        \operatorname{Vol}_h(M\setminus K_i)\to0,
\]
and
\[
        \|\psi_i^*\bar g_i-h\|_{C^0(K_i,h)}\to0.
\]
\end{theorem}

\noindent \subsection{Interpretation for the reduced Einstein dynamics}

The stability theorem gives a precise mathematical formulation of the
asymptotic picture suggested by Fischer--Moncrief
\cite{fischer2000reduced,fischer2001reduced}: on an expanding
CMC branch over a compact hyperbolizable Cauchy slice, the Lorentz cone over
the hyperbolic metric is the natural ground state of the reduced Einstein
dynamics.

\noindent Indeed, along a vacuum CMC solution one writes
\[
        k=\Sigma+\frac{\tau}{3}g,
        \qquad
        \tau<0,
        \qquad
        \operatorname{tr}_g\Sigma=0,
        \qquad
        \operatorname{div}_g\Sigma=0 .
\]
After the CMC normalization
\[
        \bar g=\frac{\tau^2}{9}g,
        \qquad
        \bar\Sigma=\frac{|\tau|}{3}\Sigma ,
\]
the constraint equations become
\[
        R(\bar g)+6=|\bar\Sigma|_{\bar g}^2,
        \qquad
        \operatorname{div}_{\bar g}\bar\Sigma=0 .
\]
Thus the deviation from the hyperbolic scalar-curvature lower bound is
exactly measured by the normalized transverse-traceless part of the second
fundamental form.

\noindent The Fischer--Moncrief reduced Hamiltonian is, up to the dimensional
normalization,
\[
        \mathcal H(g,k)=(-\tau)^3\operatorname{Vol}_g(M)
        =
        27\operatorname{Vol}_{\bar g}(M).
\]
It is monotone nonincreasing along the expanding CMC Einstein flow, and its
formal minimum in the hyperbolizable class is attained by the Lorentz cone
solution
\[
        \mathbf g_{\mathrm{LC}}
        =
        -dT^2+T^2h,
        \qquad
        \operatorname{Ric}_h=-2h .
\]
Consequently, our theorem says the following: whenever a sequence of CMC
times \(t_i\) along such a solution satisfies
\[
        \mathcal H(g(t_i),k(t_i))
        \longrightarrow
        \mathcal H(h,k^h),
\]
the corresponding normalized data converge to the Lorentz-cone data outside
sets of vanishing normalized volume.  More precisely, according to the theorem \ref{main2}, there exist good
regions \(G_i\subset M\), reference domains \(K_i\subset M\), and
diffeomorphisms
$\psi_i:K_i\to G_i$
such that
\[
        \operatorname{Vol}_{\bar g(t_i)}(M\setminus G_i)\to0,
        \qquad
        \operatorname{Vol}_h(M\setminus K_i)\to0,
\]
and
\[
        \|\psi_i^*\bar g(t_i)-h\|_{C^0(K_i,h)}\to0 .
\]
Thus, at any sequence of times along which the reduced Hamiltonian approaches
its ground-state value, the normalized spatial geometry is volume-dominated
by the hyperbolic metric, and the non-hyperbolic part of the geometry is
confined to regions of vanishing normalized volume.

\noindent It is important, however, that this conclusion is conditional on the
existence of such late-time slices.  Monotonicity of the reduced Hamiltonian
alone implies the existence of a limit
\[
        \lim_{t\to\infty}\mathcal H(g(t),k(t))
        \geq
        \mathcal H(h,k^h),
\]
provided the expanding CMC flow exists for all future time, but it does not
by itself force this limit to equal the Lorentz-cone value.  The theorem
therefore proves a rigidity and compactness statement at the bottom of the
reduced Hamiltonian, not a large-data global convergence theorem for the
vacuum Einstein equations.

The latter problem remains substantially harder.  In large data, one cannot
expect the expanding CMC foliation to persist without additional control.
Concentration of the transverse-traceless part \(\Sigma\) may lead to strong
curvature focusing and trapped-surface formation, and hence to causal
geodesic incompleteness.  This is consistent with the continuation criteria
for the Cauchy problem: in the CMC framework, breakdown can occur only if the
appropriate spacetime control norms, in particular the \(L^\infty\) control
of the second fundamental form and the gradient of the lapse, degenerate by the fundamental work of Klainerman-Rodnianski \cite{klainerman2010breakdown}.
Thus the present theorem should be viewed as a sharp stability statement for
near-minimizers of the reduced Hamiltonian.

% Preamble:
% \usepackage{amsmath}
% \usepackage{tikz}
% \usetikzlibrary{arrows.meta}

\begin{figure}
\begin{center}
\includegraphics[width=15cm,height=68cm,keepaspectratio,keepaspectratio]{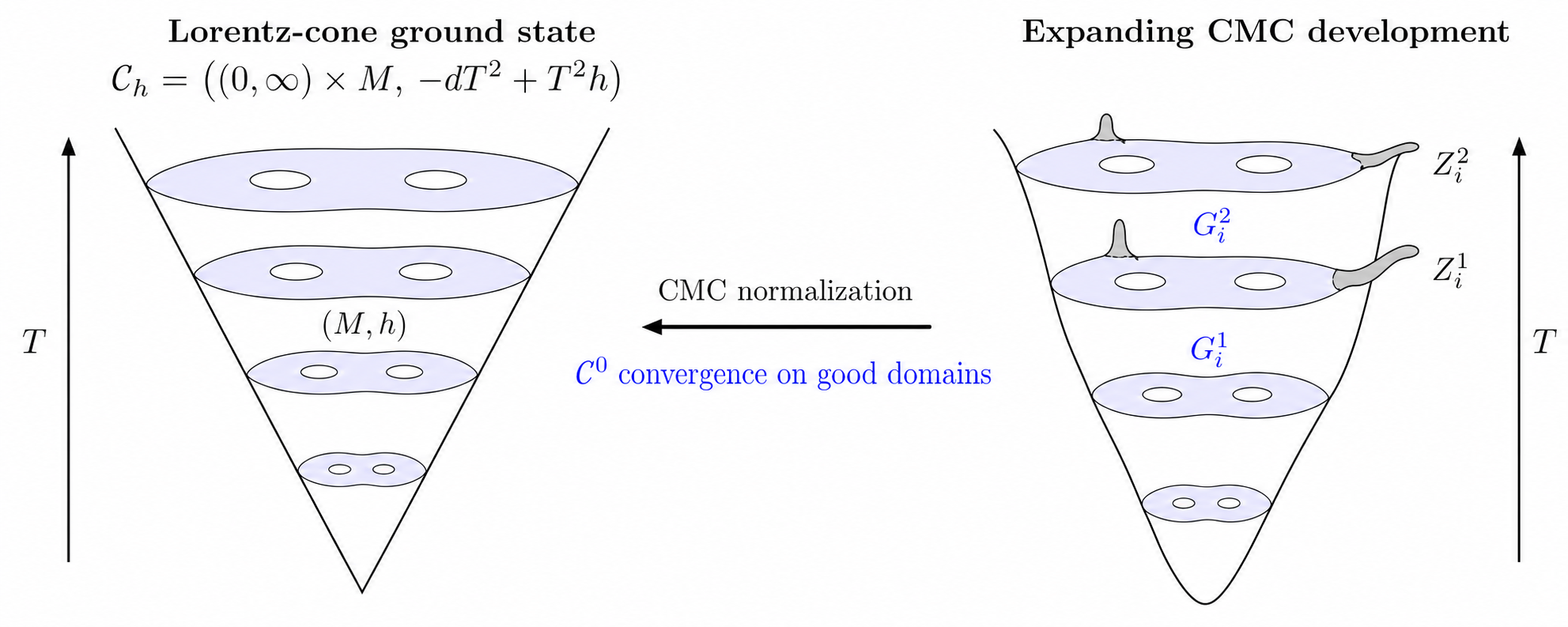}
\end{center}
\caption{
Schematic representation of reduced-Hamiltonian stability for an
expanding CMC Einstein development. Assuming global existence of the
CMC evolution, for instance through verification of an appropriate
Klainerman--Rodnianski continuation criterion, the normalized
large-time slices decompose into good domains \(G_i^{\,1}\) and
\(G_i^{\,2}\) and exceptional regions \(Z_i^{\,1}\) and \(Z_i^{\,2}\).
On the good domains, the normalized geometry converges in \(C^0\) to
the hyperbolic metric \(h\). The exceptional regions may contain fingers, but their
normalized volumes vanish asymptotically; their smaller size on the
later slice is intended to illustrate this reduced Hamiltonian decay. The limiting
ground-state spacetime is the Lorentz cone
\(\mathcal C_h=((0,\infty)\times M,-dT^2+T^2h)\).
}
\label{fig:lorentz-cone-stability}
\end{figure}

\subsection{Acknowledgements}
\noindent This work is supported by Beijing Institute of Mathematical Sciences and Applications of Yau Mathematical Science Center at Tsinghua University.

\section{Normalized Ricci Flow on Hyperbolizable Manifolds} 
\label{ricci}
\noindent Let us consider a closed connected Riemannian $3-$manifold $(M,g)$ in the same topological class as the closed connected hyperbolic manifold $(M,h)$. In addition, we always assume the scalar curvature bound 
$\text{Scal}(g)\geq -6$. We consider the normalized Ricci flow with the initial condition $(M,g)$ i.e., we aim to construct one parameter family of metrics $t\mapsto g(t)$ with $g(0)=g$ such that 
\begin{eqnarray}
\label{eq:flow}
 \frac{\partial g_{ij}}{\partial t}=-2(\text{Ric}[g]+2g(t)),~g(t=0)=g.   
\end{eqnarray}
Now define the obstruction tensor $E[g]:=\text{Ric}[g]+2g$. First, we prove a well-known but crucial proposition about the preservation of the lower bound of the scalar curvature under the normalized Ricci flow. 
\begin{proposition}
\label{scalar1}
Let \(g(t)\), \(t\in[0,T)\), be a smooth solution of the normalized Ricci flow
\[
        \partial_t g=-2\operatorname{Ric}_{g}-4g
\]
on a closed three-manifold.  Set
\[
        Q(t):=\operatorname{Scal}_{g(t)}+6,
        \qquad
        E(t):=\operatorname{Ric}_{g(t)}+2g(t).
\]
If \(Q(0)\geq 0\) on \(M\), then
\[
        Q(t)\geq 0
        \qquad\text{on }M
\]
for every \(t\in[0,T)\).  More precisely,
\[
        \min_M Q(t)\geq e^{-4t}\min_M Q(0).
\]
\end{proposition}

\begin{proof}
For the normalized flow
\[
        \partial_t g=-2\operatorname{Ric}_{g}-4g,
\]
the scalar curvature satisfies
\[
        \partial_t \operatorname{Scal}_{g}
        =
        \Delta_g \operatorname{Scal}_{g}
        +2|\operatorname{Ric}_{g}|_g^2
        +4\operatorname{Scal}_{g}.
\]
Since \(Q=\operatorname{Scal}_{g}+6\), this gives
\[
        (\partial_t-\Delta_g)Q
        =
        2|\operatorname{Ric}_{g}|_g^2
        +4\operatorname{Scal}_{g}.
\]
Writing
\[
        \operatorname{Ric}_{g}=E-2g,
        \qquad
        \operatorname{tr}_g E=\operatorname{Scal}_{g}+6=Q,
\]
and we compute
$|\operatorname{Ric}_{g}|_g^2=
        |E|_g^2-4Q+12.
$
Because \(\operatorname{Scal}_{g}=Q-6\), it follows that
\[
        (\partial_t-\Delta_g)Q
        =
        2|E|_g^2-4Q.
\]
Equivalently,
\[
        (\partial_t-\Delta_g+4)Q=2|E|_g^2\geq 0.
\]
Let
\[
        \widetilde Q(t,x):=e^{4t}Q(t,x).
\]
Then
\[
        (\partial_t-\Delta_g)\widetilde Q
        =
        e^{4t}\bigl((\partial_t-\Delta_g)Q+4Q\bigr)
        =
        2e^{4t}|E|_g^2
        \geq 0.
\]
By the parabolic maximum principle on the closed manifold \(M\),
\[
        \min_M \widetilde Q(t)\geq \min_M \widetilde Q(0)=\min_M Q(0).
\]
Hence
\[
        \min_M Q(t)\geq e^{-4t}\min_M Q(0).
\]
In particular, if \(Q(0)\geq0\), then \(Q(t)\geq0\) for every \(t\in[0,T)\).
\end{proof}
\noindent This proposition \ref{scalar1} is vital and allows us to utilize the normalized Ricci flow technique to study the volume stability. In the following straightforward proposition, we study how the spacetime integral of the scalar curvature defect $Q$ is controlled by means of the volume deficit. 
\begin{proposition}
\label{scalar2}
Let \(g(t)\), \(t\in[0,T)\), be a smooth solution of the normalized Ricci flow
\[
        \partial_t g=-2\operatorname{Ric}_{g}
        -4g
\]
on a closed three-manifold \(M\).  Set
\[
        Q(t):=\operatorname{Scal}_{g(t)}+6.
\]
Then, for every \(0\leq \tau<T\),
\[
        \int_0^\tau\int_M Q(t)\,d\mu_{g(t)}\,dt
        =
        \operatorname{Vol}_{g(0)}(M)-\operatorname{Vol}_{g(\tau)}(M).
\]
In particular, if \(g(t)\) extends smoothly to \(t=T\), the same identity holds with \(\tau=T\).
\end{proposition}

\begin{proof}
The first variation formula for the Riemannian volume form gives
\[
        \partial_t d\mu_{g(t)}
        =
        \frac12\operatorname{tr}_{g(t)}(\partial_t g)\,d\mu_{g(t)}.
\]
Using the normalized Ricci flow equation,
\[
        \operatorname{tr}_{g(t)}(\partial_t g)
       =
        -2Q(t).
\]
Hence
\[
        \partial_t d\mu_{g(t)}=-Q(t)\,d\mu_{g(t)}.
\]
Since \(M\) is closed and the solution is smooth, differentiation under the integral is justified, and therefore
\[
        \frac{d}{dt}\operatorname{Vol}_{g(t)}(M)
        =
        \int_M \partial_t d\mu_{g(t)}
        =
        -\int_M Q(t)\,d\mu_{g(t)}.
\]
Integrating this identity from \(0\) to \(\tau<T\) yields
\[
        \operatorname{Vol}_{g(\tau)}(M)-\operatorname{Vol}_{g(0)}(M)
        =
        -\int_0^\tau\int_M Q(t)\,d\mu_{g(t)}\,dt,
\]
which is the desired formula.
\end{proof}\noindent Now we introduce the most important entity in our context namely the Perelman entropy functional. For a closed Riemannian $3-$manifold, define the Perelman's entropy functional as 
\begin{eqnarray}
\lambda[g]:=\inf_{\int_{M}u^{2}d\mu_{g}=1}\int_{M}\left(4|\nabla u|^{2}_{g}+\text{Scal}(g)u^{2}\right)d\mu_{g}.    
\end{eqnarray}
Notice that this function $\lambda(g)$ is not scale-invariant. We recall the following scale-invariant $\lambda$ functional 
\begin{eqnarray}
\overline{\lambda}[g]:=\lambda[g]\text{Vol}(M,g)^{\frac{2}{3}}
\end{eqnarray}
The first result that we prove is the smallness of the entropy deficit given the small volume deficit and the scalar curvature bound.
\noindent In the next proposition, we prove that along the normalized Ricci flow, we have the following monotonicity property of the rescaled functional $\overline{\lambda}[g(t)]$

\begin{proposition}
\label{monotonicity}
Let \(M^{3}\) be a closed hyperbolizable three-manifold, and let
\(g(t)\), \(t\in[0,T)\), be a smooth solution of the rescaled Ricci flow
\begin{equation}
\label{eq:rescaled-RF}
        \partial_t g=-2\bigl(\operatorname{Ric}_{g}+2g\bigr).
\end{equation}
Let
\begin{equation}
\label{eq:lambda-functional}
        \lambda[g]
        :=
        \inf_{\int_M u^2\,d\mu_g=1}
        \int_M
        \left(
                4|\nabla u|_g^2+R_g u^2
        \right)
        d\mu_g
\end{equation}
be Perelman's \(\lambda\)-functional, and define the scale-invariant entropy
\begin{equation}
\label{eq:scaled-lambda}
        \overline{\lambda}[g]
        :=
        \lambda[g]\operatorname{Vol}(M,g)^{2/3}.
\end{equation}
Then \(\overline{\lambda}[g(t)]\) is monotonically non-decreasing along
\eqref{eq:rescaled-RF}. More precisely, if \(f=f(t)\) is the minimizer for
\(\lambda[g(t)]\), normalized by
\begin{equation}
\label{eq:f-normalization}
        \int_M e^{-f}\,d\mu_g=1,
\end{equation}
then
\begin{equation}
\label{eq:scaled-lambda-monotonicity}
\begin{aligned}
        \frac{d}{dt}\overline{\lambda}[g(t)]
        &=
        2\operatorname{Vol}(M,g)^{2/3}
        \int_M
        \left|
                \operatorname{Ric}_g+\nabla^2 f
                -
                \frac{\lambda[g]}{3}g
        \right|_g^2
        e^{-f}\,d\mu_g        \\
        &\quad
        +
        \frac{2}{3}\operatorname{Vol}(M,g)^{2/3}
        \lambda[g]\bigl(\lambda[g]-\bar R_g\bigr),
\end{aligned}
\end{equation}
where
\begin{equation}
\label{eq:average-scalar-curvature}
        \bar R_g
        :=
        \frac{1}{\operatorname{Vol}(M,g)}
        \int_M R_g\,d\mu_g .
\end{equation}
In particular,
\begin{equation}
\label{eq:monotonic1}
        \frac{d}{dt}\overline{\lambda}[g(t)]\geq 2\operatorname{Vol}(M,g)^{2/3}
        \int_M
        \left|
                \operatorname{Ric}_g+\nabla^2 f
                -
                \frac{\lambda[g]}{3}g
        \right|_g^2
        e^{-f}\,d\mu_g  \geq 0.
\end{equation}
\end{proposition}

\begin{proof}
Let
\[
        V(t):=\operatorname{Vol}(M,g(t))
\]
denote the volume of $(M,g(t))$.
Since
\[
        \partial_t g=-2\operatorname{Ric}_g-4g,
\]
we have
\[
        \operatorname{tr}_g(\partial_t g)
        =
        -2R_g-12.
\]
Therefore, direct computation yields
\begin{equation}
\label{eq:measure-evolution-rescaled}
        \partial_t d\mu_g
        =
        \frac{1}{2}\operatorname{tr}_g(\partial_t g)d\mu_g
        =
        -(R_g+6)d\mu_g
\end{equation}
leading to
\begin{equation}
\label{eq:volume-evolution-rescaled}
        \frac{d}{dt}V(t)
        =
        -\int_M (R_g+6)\,d\mu_g
        =
        -V(t)(\bar R_g+6),
\end{equation}
where denote the average scalar curvature over $(M,g)$ by $\overline{R}_{g}$. Next we compute the evolution of \(\lambda[g(t)]\). Instead of direct computation, we observe the following point regarding the scaling property of the $\lambda(g)$ functional. Along the ordinary
unnormalized Ricci flow
\[
        \partial_t g=-2\operatorname{Ric}_g,
\]
Perelman's monotonicity formula gives
\begin{equation}
\label{eq:perelman-lambda-ordinary}
        \frac{d}{dt}\lambda[g(t)]
        =
        2\int_M
        \left|
                \operatorname{Ric}_g+\nabla^2 f
        \right|_g^2
        e^{-f}\,d\mu_g .
\end{equation}
The present flow differs from the ordinary Ricci flow by the pure scaling
term \(-4g\). We use the scaling to obtain the time evolution of $\lambda[g(t)]$ along the normalized Ricci flow Since
\[
        \lambda[cg]=c^{-1}\lambda[g],
\]
the pure scaling flow \(\partial_t g=-4g\) contributes \(4\lambda[g]\) to
\(\frac{d}{dt}\lambda[g(t)]\). Hence, along
\eqref{eq:rescaled-RF}, one has
\begin{equation}
\label{eq:lambda-evolution-rescaled}
        \frac{d}{dt}\lambda[g(t)]
        =
        2\int_M
        \left|
                \operatorname{Ric}_g+\nabla^2 f
        \right|_g^2
        e^{-f}\,d\mu_g
        +
        4\lambda[g(t)].
\end{equation}
One can also see this by direct differentiation, using the normalized Ricci flow, and the fact that $\text{Ric}+\text{Hess}(f)$ is divergence free with respect to the measure $e^{-f}d\mu_{g}$. In particular, one recalls the expression for $\lambda$ in terms of $f$ for which the infimum of $\int_{M}(4|\nabla u|^{2}+R_{g}u^{2})d\mu_{g}$ is attained with $u=e^{-f/2}$ and $\int_{M}u^{2}d\mu_{g}=1$
\begin{align}
    \lambda(g)=2\Delta_{g}f-|\nabla f|^{2}_{g}+R_{g}.
\end{align}
Clearly,  $\lambda(g)$ is not monotonic along the normalized Ricci flow since for a hyperbolizable manifold $\lambda(g)$ is negative. Therefore, we turn our attention to the scale-invariant entropy functional, which is monotonic under the re-scaled Ricci flow.   

\noindent We now differentiate
\[
        \overline{\lambda}[g(t)]
        =
        \lambda[g(t)]V(t)^{2/3}.
\]
Using \eqref{eq:volume-evolution-rescaled} and
\eqref{eq:lambda-evolution-rescaled}, we obtain
\begin{align}
        \frac{d}{dt}\overline{\lambda}[g(t)]
        &=
        V^{2/3}\frac{d}{dt}\lambda[g(t)]
        +
        \frac{2}{3}\lambda[g(t)]V^{-1/3}\frac{dV}{dt} \notag \\
        &=
        V^{2/3}
        \left[
                2\int_M
                \left|
                        \operatorname{Ric}_g+\nabla^2 f
                \right|_g^2
                e^{-f}\,d\mu_g
                +
                4\lambda[g]
        \right]
        -
        \frac{2}{3}\lambda[g]V^{2/3}(\bar R_g+6)\\
       & =
        2V^{2/3}
        \int_M
        \left|
                \operatorname{Ric}_g+\nabla^2 f
        \right|_g^2
        e^{-f}\,d\mu_g
        -
        \frac{2}{3}V^{2/3}\lambda[g]\bar R_g.
\end{align}
It remains to rewrite this expression in a nonnegative form for the case of a hyperbolizable manifold that we are considering. The following series of computations yield the desired result. Since \(f\)
realizes \(\lambda[g]\), we have
\begin{equation}
\label{eq:lambda-f-form}
        \lambda[g]
        =
        \int_M
        \left(
                R_g+|\nabla f|_g^2
        \right)
        e^{-f}\,d\mu_g .
\end{equation}
Moreover, note the following integration by parts,
\begin{equation}
\label{eq:laplacian-f-identity}
        \int_M \Delta f\,e^{-f}\,d\mu_g
        =
        \int_M |\nabla f|_g^2 e^{-f}\,d\mu_g .
\end{equation}
Therefore
\begin{equation}
\label{eq:trace-identity}
        \int_M
        \operatorname{tr}_g
        \left(
                \operatorname{Ric}_g+\nabla^2 f
        \right)
        e^{-f}\,d\mu_g
        =
        \int_M
        (R_g+\Delta f)e^{-f}\,d\mu_g
        =
        \lambda[g].
\end{equation}
Since \(\dim M=3\), we decompose
\[
        \operatorname{Ric}_g+\nabla^2 f
        =
        \left(
                \operatorname{Ric}_g+\nabla^2 f
                -
                \frac{\lambda[g]}{3}g
        \right)
        +
        \frac{\lambda[g]}{3}g.
\]
Using \eqref{eq:trace-identity}, we get
\begin{equation}
\label{eq:square-decomposition}
\begin{aligned}
        \int_M
        \left|
                \operatorname{Ric}_g+\nabla^2 f
        \right|_g^2
        e^{-f}\,d\mu_g
        &=
        \int_M
        \left|
                \operatorname{Ric}_g+\nabla^2 f
                -
                \frac{\lambda[g]}{3}g
        \right|_g^2
        e^{-f}\,d\mu_g        \\
        &\quad
        +
        \frac{\lambda[g]^2}{3}.
\end{aligned}
\end{equation}
Substituting \eqref{eq:square-decomposition} into the time derivative of the rescaled entropy $\overline{\lambda}[g(t)]$ gives
\begin{align}
        \frac{d}{dt}\overline{\lambda}[g(t)]
        &=
        2V^{2/3}
        \int_M
        \left|
                \operatorname{Ric}_g+\nabla^2 f
                -
                \frac{\lambda[g]}{3}g
        \right|_g^2
        e^{-f}\,d\mu_g
        +
        \frac{2}{3}V^{2/3}\lambda[g]^2
        -
        \frac{2}{3}V^{2/3}\lambda[g]\bar R_g \notag \\
        &=
        2V^{2/3}
        \int_M
        \left|
                \operatorname{Ric}_g+\nabla^2 f
                -
                \frac{\lambda[g]}{3}g
        \right|_g^2
        e^{-f}\,d\mu_g
        +
        \frac{2}{3}V^{2/3}
        \lambda[g]\bigl(\lambda[g]-\bar R_g\bigr).
\end{align}
This is precisely \eqref{eq:scaled-lambda-monotonicity}. Testing the definition of \(\lambda[g]\) with
the constant function \(u=Vol^{-1/2}\), we obtain
\begin{equation}
\label{eq:lambda-less-average-scalar}
        \lambda[g]\leq \bar R_g .
\end{equation}
Since \(M\) is closed and hyperbolizable, its Yamabe invariant is negative.
In particular \(M\) admits no metric of positive scalar curvature, and hence
\begin{equation}
\label{eq:lambda-nonpositive}
        \lambda[g]\leq 0
\end{equation}
for every smooth metric \(g\) on \(M\). Combining
\eqref{eq:lambda-less-average-scalar} and \eqref{eq:lambda-nonpositive},
we get
\[
        \lambda[g]\bigl(\lambda[g]-\bar R_g\bigr)\geq 0.
\]
The first term in \eqref{eq:scaled-lambda-monotonicity} is manifestly
nonnegative. Therefore
\[
        \frac{d}{dt}\overline{\lambda}[g(t)]\geq 0.
\]
This proves the proposition.
\end{proof}

\begin{corollary}
The volume normalized entropy $\overline{\lambda}(g(t))$ is non-decreasing along the volume normalized Ricci flow as long as the latter exists. Moreover, it is constant if and only if $(M,g)$ is compact hyperbolic.  
\end{corollary}
\begin{proof}
Let \(n=3\), and write
\[
        \overline{\lambda}(g)
        :=
        \lambda(g)\operatorname{Vol}(g)^{2/3}.
\]
Recall by the previous proposition \ref{monotonicity}
\[
\begin{aligned}
        \frac{d}{dt}\overline{\lambda}(g(t))
        &=
        2V^{2/3}
        \int_M
        \left|
        \operatorname{Ric}
        +
        \nabla^2 f
        -
        \frac{\lambda}{3}g
        \right|^2 e^{-f}\,d\mu_g        \\
        &\quad+
        \frac{2}{3}V^{2/3}\lambda
        \left(
        \lambda-\frac{1}{V}\int_M R\,d\mu_g
        \right),
\end{aligned}
\]
where \(V=\operatorname{Vol}(M,g(t))\). Since the constant function is an admissible test function for
\(\lambda(g)\),
\[
        \lambda(g)
        \leq
        \frac{1}{V}\int_M R\,d\mu_g .
\]
On a negative Yamabe manifold one has \(\lambda(g)\leq0\) for every metric
\(g\). Hence
\[
        \lambda
        \left(
        \lambda-\frac{1}{V}\int_M R\,d\mu_g
        \right)
        \geq0.
\]
Thus
\[
        \frac{d}{dt}\overline{\lambda}(g(t))\geq0.
\]
Since \(\overline{\lambda}\) is scale-invariant, the same monotonicity holds
along the volume-normalized Ricci flow. If \(\overline{\lambda}\) is constant on a time interval, then both
nonnegative terms above vanish. Hence
\[
        \operatorname{Ric}+\nabla^2 f
        =
        \frac{\lambda}{3}g
\]
and, since \(\lambda<0\) in the negative Yamabe class,
\[
        \lambda=
        \frac{1}{V}\int_M R\,d\mu_g .
\]
The latter equality means that the constant test function realizes the
minimum in the definition of \(\lambda(g)\). Therefore \(f\) is constant.
Consequently
\[
        \operatorname{Ric}
        =
        \frac{\lambda}{3}g .
\]
Thus \(g\) is Einstein with negative Einstein constant. In dimension three,
a negative Einstein metric has constant sectional curvature. Hence \(g\) is
compact hyperbolic.

\noindent Conversely, if \(g\) is compact hyperbolic, then it is Einstein and is fixed
by the volume-normalized Ricci flow up to the chosen normalization. Hence
\(\overline{\lambda}(g(t))\) is constant.
\end{proof}

\begin{remark}
The scale-invariant Perelman entropy
\[
        \overline{\lambda}[g]
        :=
        \lambda[g]\operatorname{Vol}(g)^{2/3}
\]
is monotone along Ricci flow only on the branch where
\(\lambda[g]\leq 0\).  In particular, on a negative Yamabe
three-manifold one has \(\lambda[g]\leq0\) for every metric \(g\), and hence
\(\overline{\lambda}[g(t)]\) is monotone along the flow.  In the positive
Yamabe case this sign condition may fail; when \(\lambda[g(t)]>0\), the
monotonicity of \(\overline{\lambda}\) is no longer guaranteed.  Thus
\(\overline{\lambda}\) cannot be used as a global monotone quantity in the
positive Yamabe setting.
\end{remark}

\noindent In order to use the entropy monotonicity, we first need to prove how the entropy deficit is controlled by the volume deficit. We show it in the next proposition 
\begin{proposition}
\label{entropydeficit}
 Let $(M,g)$ be a closed $3-$ manifold with $R(g)\geq -6$ and let $(M,h)$ be hyperbolic manifold with $R(h)=-6$. Suppose $V_{g}=V_{h}+\delta$, then the following is verified by the scale-invariant entropy 
 \begin{eqnarray}
  0\leq \lambda(h)-\lambda(g)  \leq C_{h}\delta 
 \end{eqnarray}
 for $\delta\ll 1$ and $C_{h}$ depending on the hyperbolic metric $h$.
\end{proposition}
\begin{proof}
First, recall the definition of the entropy $\lambda(g)$ with $u=e^{-f}$
\begin{eqnarray}
 \lambda(g)=\inf_{\int_{M}u^{2}d\mu_{g}=1}(R_{g}u^{2}+4|\nabla u|^{2}_{g})d\mu_{g}   
\end{eqnarray}
which in light of the scalar curvature estimate $R_{g}$ yields 
\begin{align}
  \lambda(g)\geq -6  
\end{align}
and therefore 
\begin{eqnarray}
  \bar{\lambda}(g)\geq -6(V_{g})^{\frac{2}{3}}=-6(V_{h}+\delta)^{\frac{2}{3}}.  
\end{eqnarray}
Now since \(M\) is hyperbolic, its Yamabe invariant is negative, and the hyperbolic metric realizes the maximal value of \(\bar\lambda\) i.e., we have the upper bound 
\begin{align}
    \bar\lambda(g)\leq \bar\lambda(h)=-6(V_{h})^{\frac{2}{3}}. 
\end{align}
Therefore we have 
\begin{align}
    0\leq \bar\lambda(h)-\bar\lambda(g)\leq 6\bigg((V_{h}+\delta)^{\frac{2}{3}}-(V_{h})^{\frac{2}{3}}\bigg)\leq C_{h}\delta 
\end{align}
for $\delta\ll 1$.
\end{proof}

\begin{proposition}
\label{integraldefect1}
 Let the volume-normalized Ricci flow exist on $[0,T]$. Then the following spacetime estimate holds on any regular interval $[0,T)$   
 \begin{eqnarray}
  \bigg|\int_{0}^{T}\left(\text{Vol}(g(t))^{\frac{2}{3}}\int_{M}\bigg|\text{Ric}-\frac{\lambda}{3}g+\nabla^{2}f\bigg|^{2}e^{-f}d\mu_{g}\right)dt\bigg|\lesssim \delta
 \end{eqnarray}
\end{proposition}

\begin{proof}
First recall that the entropy deficit $\mathcal{D}(t)=\overline{\lambda}(h)-\overline{\lambda}(g(t))$ is non-increasing along the normalized Ricci flow $t\mapsto g(t)$ as long as the flow exists and therefore, 
\begin{eqnarray}
\mathcal{D}(t)\leq \mathcal{D}(0)\lesssim \delta. 
\end{eqnarray}
Now integration of the inequality \ref{eq:monotonic1} yields
  \begin{eqnarray}
  \int_{0}^{T}\left(\text{Vol}(g(t))^{\frac{2}{3}}\int_{M}\bigg|\text{Ric}-\frac{\lambda}{3}g+\nabla^{2}f\bigg|^{2}e^{-f}d\mu_{g}\right)dt\leq \frac{1}{2}(\overline{\lambda}(g(T))-\overline{\lambda}(g(0)))\\\nonumber =\frac{1}{2}(\mathcal{D}(0)-\mathcal{D}(T))\leq \frac{1}{2}\mathcal{D}(0)\lesssim \delta
 \end{eqnarray}  
 
\end{proof}
\noindent Now we recall the following elementary averaging lemma that will allow us to pick a good time slice $(M,g(t))$ in the Ricci flow developement. 
\begin{lemma}
\label{lem:averaging-lemma}
Let \(I=[a,b]\subset \mathbb R\) be a compact interval of positive length, and let
$  A:I\longrightarrow [0,\infty]$
be a measurable function. Suppose that
$
        \int_a^b A(t)\,dt \leq D
$
for some \(D\geq 0\). Then, there exists a time \(t_0\in I\) such that
$A(t_0)
        \leq
        \frac{D}{b-a}$.
More generally, for every \(\alpha>1\), the set
\begin{equation}
\label{eq:good-time-set}
        \mathcal G_{\alpha}
        :=
        \left\{
        t\in I:\ 
        A(t)\leq
        \alpha\,\frac{D}{b-a}
        \right\}
\end{equation}
satisfies
\begin{equation}
\label{eq:good-time-set-measure}
        |\mathcal G_{\alpha}|
        \geq
        \left(1-\frac{1}{\alpha}\right)(b-a).
\end{equation}
\end{lemma}

\begin{proof} This lemma is a straightforward consequence of the Chebyshev's inequality on $L^{1}$ and contradiction argument.
Since \(A\geq 0\), Chebyshev's inequality gives
\[
        \theta
        \left|
        \left\{
        t\in I:\ A(t)>\theta
        \right\}
        \right|
        \leq
        \int_{\{A>\theta\}}A(t)\,dt
        \leq
        \int_a^b A(t)\,dt
        \leq D.
\]
Choosing
$\theta=\alpha\,\frac{D}{b-a}$
with \(\alpha>1\), we obtain
\[
        |I\setminus \mathcal G_{\alpha}|
        =
        \left|
        \left\{
        t\in I:\ 
        A(t)>
        \alpha\,\frac{D}{b-a}
        \right\}
        \right|
        \leq
        \frac{D}{\alpha D/(b-a)}
        =
        \frac{b-a}{\alpha}.
\]
Therefore
\[
        |\mathcal G_{\alpha}|
        \geq
        (b-a)-\frac{b-a}{\alpha}
        =
        \left(1-\frac{1}{\alpha}\right)(b-a),
\]
which proves \eqref{eq:good-time-set-measure}. Finally, letting \(\alpha\downarrow 1\), or arguing directly by contradiction, gives the existence of \(t_0\in I\) satisfying $A(t_0)
        \leq
        \frac{D}{b-a}$. Indeed, if
$ A(t)>\frac{D}{b-a}$
for almost every \(t\in I\), then
\[
        \int_a^b A(t)\,dt>D,
\]
contradicting $\int_{a}^{b}A(t)dt\leq D$. This completes the proof.
\end{proof}
\noindent Using this elementary lemma, we obtain the following corollary of the proposition \ref{integraldefect1} namely we can pick out a good time slice where the weighted $L^{2}$ norm of the soliton defect $\text{Ric}+\nabla^{2}f-\frac{\lambda(g)}{3}g$ is small. 
\begin{corollary}
\label{cor:perelman-good-time}
Let \(M^{3}\) be a closed hyperbolizable manifold and let
\[
        \partial_t g_i=-2\bigl(\operatorname{Ric}_{g_i}+2g_i\bigr),
        \qquad t\in[0,T],
\]
be a sequence of smooth rescaled Ricci flows. Let \(f_i(t)\) be the minimizer
for Perelman's \(\lambda\)-functional, normalized by $\int_M e^{-f_i(t)}\,d\mu_{g_i(t)}=1$. 
Whenever
\begin{equation}
\label{eq:entropy-deficit-bound}
        \overline{\lambda}[h]-\overline{\lambda}[g_i(0)]
        \lesssim
        \delta_i,
        \qquad
        \delta_i\to 0,
\end{equation}
where \(h\) is the hyperbolic metric normalized by
\[
        \operatorname{Ric}_h=-2h,
\]
for every fixed \(T>0\), there exists a time $t_i\in \left[\frac{T}{2},T\right]$
such that
\begin{equation}
\label{eq:perelman-good-time-estimate}
 \operatorname{Vol}(M,g_i(t_i))^{2/3}
        \int_M
        \left|
        \operatorname{Ric}_{g_i(t_i)}
        +
        \nabla^2 f_i(t_i)
        -
        \frac{\lambda[g_i(t_i)]}{3}g_i(t_i)
        \right|_{g_i(t_i)}^2
        e^{-f_i(t_i)}\,d\mu_{g_i(t_i)}
        \lesssim
        \frac{1}{T}\,\delta_i.
\end{equation}
\end{corollary}

\begin{proof} The proof is a direct consequence of the monotonicity property of the scale-invariant entropy functional $\overline{\lambda}$ as proven in proposition \ref{eq:monotonic1} and the integral estimate of corollary \ref{integraldefect1}. Suppose for notational convenience, we denote the integrated soliton defect as follows 
\[
A_{i}:=\int_M
        \left|
        \operatorname{Ric}_{g_i(t_i)}
        +
        \nabla^2 f_i(t_i)
        -
        \frac{\lambda[g_i(t_i)]}{3}g_i(t_i)
        \right|_{g_i(t_i)}^2
        e^{-f_i(t_i)}\,d\mu_{g_i(t_i)}
\]

\noindent It follows from corollary \ref{integraldefect1}
\begin{align}
        2\int_{T/2}^{T}A_i(t)\,dt
        &\leq
        \overline{\lambda}[g_i(T)]
        -
        \overline{\lambda}[g_i(T/2)]                                    \\
        &\leq
        \overline{\lambda}[h]
        -
        \overline{\lambda}[g_i(0)]                                      \\
        &\lesssim
        \delta_i.
\end{align}
Applying Lemma \ref{lem:averaging-lemma} to the interval $I=\left[\frac{T}{2},T\right]$,
gives a time \(t_i\in[T/2,T]\) satisfying
\[
        A_i(t_i)
        \lesssim
        \frac{\delta_i}{T}.
\]
This proves the corollary.
\end{proof}
\noindent Next, we recall the following vital theorem regarding the Ricci flow with Surgery (see \cite{perelman2003ricci}, the exposition by \cite{cao2006complete} for the detailed account)
\begin{theorem}
Let (M; g) be a closed, 3-
dimensional Riemannian manifold. If the
surgery scales $r_{surg}(T_{i}) > 0$ are chosen sufficiently small (depending on $(M,g)$ and $T_{i}$), then a Ricci 
flow with surgery with initial condition $(M_{1},g_{1}=M_{0},g(0))$ can be constructed. 
\end{theorem}
\noindent A central point in the construction of the three-dimensional Ricci flow with
surgery is to exclude the possibility that the surgery times accumulate in a
finite time interval.  This is not merely a bookkeeping issue.  At each surgery
time, one replaces a sufficiently round neck with a pair of standard caps, and this
operation introduces a small geometric error.  If the surgery parameters were
kept fixed, these errors could, in principle, accumulate and force a Zeno-type
phenomenon, namely infinitely many surgeries before some finite time \(T\).

\noindent Perelman's resolution is to choose the surgery parameters with increasing
precision as time evolves.  Thus the necks along which the cuts are made are
required to be increasingly close, after scaling, to the standard round cylinder,
and the caps inserted after the cut are correspondingly close to the standard
solution.  In particular, the surgery scale is chosen so small that, after
rescaling about a region of controlled curvature, the actual surgery region is
pushed arbitrarily far away in space.  This allows one to recover, on compact
subsets of the rescaled spacetime, a genuinely smooth Ricci flow rather than a
flow interrupted by surgery.

\noindent There is, however, a delicate analytic point.  Hamilton's compactness theorem is
a theorem for smooth Ricci flows.  A Ricci flow with surgery is only piecewise
smooth in time, and therefore one cannot directly apply Hamilton compactness to
a sequence of surgically modified flows.  Moreover, a uniform \(C^{0}\)-bound for
the curvature is not by itself sufficient: in order to pass to a smooth limit one
also needs the higher derivative bounds supplied by Shi's estimates \cite{shi89}, and these
estimates require an honest smooth solution on a definite time interval.

\noindent The required remedy has two components that we discuss here.  First, Perelman's choice of sufficiently
fine necks and a sufficiently small cutoff scale ensure that, under the relevant
parabolic rescalings, the regions altered by surgery escape every fixed compact
subset.  Second, one proves a local time-extension statement for the surgically
modified solution: away from the newly inserted caps, and on compact subsets
which remain at a definite distance from the surgery region, the post-surgery flow
is smooth on a uniform short time interval.  Consequently, Shi's local derivative
estimates \cite{shi89}
apply to these compact spacetime regions, and Hamilton's compactness
theorem may be used exactly as in the smooth case.

\noindent Schematically, the argument is
fine surgery
$\rightarrow$
surgery region escapes under rescaling
$\rightarrow$
local smooth time extension, 
and hence
Shi estimates
$\rightarrow$
Hamilton compactness
$\rightarrow$
smooth canonical limit. If surgery times accumulate at a finite time, one could rescale about the
corresponding high-curvature regions and extract a smooth limiting ancient
solution.  The canonical-neighborhood structure of this limit is then
incompatible with the assumed accumulation of surgery times.  This contradiction
shows that the surgery times are discrete, and in particular that only finitely
many surgeries occur on every compact time interval.

\noindent A further subtlety concerns the evolution of the caps inserted at surgery.  It
is not enough to know that the caps are initially close to the standard cap; one
must also know that this closeness persists for a uniform amount of time after
the surgery.  This is the role of the prolongation analysis of fine caps.  Using
the uniqueness theory for complete noncompact Ricci flows, one compares the
short-time evolution of the inserted caps with the corresponding standard cap
solution.  This gives the necessary local smoothness and derivative control near
the post-surgery region, and thereby completes the justification of the
compactness argument for surgically modified flows.
\noindent  Now we need to prove that on the good time slice $t_{i}$ ($t_{i}\to\infty$), there is a good part that is $C^{0}$ close to the final hyperbolic development. This part is crucial. This is the second part of the proof. The following proposition accomplishes this result. 
We denote the the hyperbolic manifold normalized to have scalar curvature $-6$ by $(M,h)$ and the corresponding volume is denoted by $\text{Vol}_{h}$.

\begin{proposition}
\label{prop:entropy-good-time-c0-hyperbolic-borel}
Let \((M,h)\) be a closed hyperbolic three-manifold normalized by
\[
        \operatorname{Ric}_h=-2h,
        \qquad
        R(h)=-6.
\]
For each \(i\), let \((M_i(t),g_i(t))\), \(t\geq0\), be the distinguished
component of a normalized Ricci flow with sufficiently precise surgery
starting from \((M,g_i)\), and assume
\[
        R(g_i)\geq-6,
        \qquad
        0\leq\operatorname{Vol}_{g_i}(M)-\operatorname{Vol}_h(M)
        \leq\delta_i,
        \qquad
        \delta_i\to0.
\]
Then, after passing to a subsequence, there exist \(A_i\to\infty\),
regular times
$ t_i\in[A_i,2A_i]$,
compact domains \(K_i^t\subset M\) and \(G_i^t\subset M_i(t_i)\),
diffeomorphisms
\[
        \psi_i^t:K_i^t\longrightarrow G_i^t,
\]
and bad sets \(Z_i^t:=M_i(t_i)\setminus G_i^t\) such that
\[
        \operatorname{Vol}_{g_i(t_i)}(Z_i^t)\to0,
        \qquad
        \operatorname{Vol}_h(M\setminus K_i^t)\to0,
\]
and
\[
        \|(\psi_i^t)^*g_i(t_i)-h\|_{C^0(K_i^t,h)}\to0.
\]
Moreover,
\[
        \int_{G_i^t}
        \bigl(R(g_i(t_i))+6\bigr)\,d\mu_{g_i(t_i)}
        \to0.
\]
\end{proposition}

\begin{proof}
First let us fix the notation. For each \(i\), let
\[
        \bigl(M_i(t),g_i(t)\bigr),\qquad t\geq0,
\]
denote the distinguished component of the normalized Ricci flow with
sufficiently precise surgery starting from \((M,g_i)\). Thus \(M_i(t)\)
is the component carrying the hyperbolic prime factor and is
diffeomorphic to \(M\) at every regular time. Set
\[
        V_h:=\operatorname{Vol}_h(M),
        \qquad
        V_i(t):=\operatorname{Vol}_{g_i(t)}(M_i(t)),
\]
and
\[
        Q_i(t):=R(g_i(t))+6,
        \qquad
        E_i(t):=\operatorname{Ric}_{g_i(t)}+2g_i(t).
\]

\noindent On every regular time interval,
\[
        (\partial_t-\Delta_{g_i(t)}+4)Q_i
        =
        2|E_i|_{g_i(t)}^2.
\]
Hence Proposition~\ref{scalar1} and the parabolic maximum principle give
\[
        Q_i(t)\geq0
\]
at every regular time. The standard sufficiently precise surgery
construction preserves the same inequality across surgery times: the
metric is unchanged outside the surgery regions, while the inserted
caps have positive scalar curvature at the surgery scale. Consequently,
\[
        R(g_i(t))+6\geq0
\]
on every regular and every post-surgery time slice.

\noindent Let \(\mathcal T_i\) denote the discrete set of surgery times and set
\[
        \Delta_i^{\mathrm{surg}}V(s)
        :=
        V_i(s^-)-V_i(s^+)\geq0,
        \qquad s\in\mathcal T_i.
\]
On each regular time interval,
\[
        \frac{d}{dt}V_i(t)
        =
        -\int_{M_i(t)}Q_i(t)\,d\mu_{g_i(t)}.
\]
Since the distinguished component remains in the fixed hyperbolic
topological class, the hyperbolic volume comparison theorem gives
\[
        V_h\leq V_i(t)\leq V_i(0)\leq V_h+\delta_i.
\]
Summing the volume identity over the regular time intervals, we obtain,
for every regular \(T>0\),
\[
\begin{aligned}
        V_i(0)-V_i(T)
        &=
        \int_0^T\int_{M_i(t)}
        Q_i(t)\,d\mu_{g_i(t)}\,dt
        +
        \sum_{\substack{s\in\mathcal T_i\\0<s<T}}
        \Delta_i^{\mathrm{surg}}V(s).
\end{aligned}
\]
Here and below, spacetime integrals are understood to be taken over the
regular portions of the surgical flow. Since all terms on the
right-hand side are nonnegative and \(V_i(T)\geq V_h\) (by Schoen's volume comparison result), it follows that
\begin{equation}
\label{eq:late-scalar-budget}
        \int_0^\infty\int_{M_i(t)}
        Q_i(t)\,d\mu_{g_i(t)}\,dt
        \leq\delta_i,
\end{equation}
and, simultaneously,
\[
        \sum_{s\in\mathcal T_i}
        \Delta_i^{\mathrm{surg}}V(s)
        \leq\delta_i.
\]

\noindent
We next record the corresponding entropy estimate. Let \(f_i(t)\) be
the normalized minimizer for Perelman's \(\lambda\)-functional on the
regular slice \((M_i(t),g_i(t))\),
\[
        \int_{M_i(t)}e^{-f_i(t)}\,d\mu_{g_i(t)}=1,
\]
and define
\[
        \mathcal S_i(t)
        :=
        \operatorname{Ric}_{g_i(t)}
        +
        \nabla^2_{g_i(t)}f_i(t)
        -
        \frac{\lambda[g_i(t)]}{3}g_i(t).
\]
By Proposition~\ref{entropydeficit},
\[
        0\leq
        \overline\lambda[h]-\overline\lambda[g_i(0)]
        \leq C_h\delta_i.
\]
On every regular interval, Proposition~\ref{monotonicity} gives
\[
        \frac{d}{dt}\overline\lambda[g_i(t)]
        \geq
        2V_i(t)^{2/3}
        \int_{M_i(t)}
        |\mathcal S_i(t)|_{g_i(t)}^2
        e^{-f_i(t)}\,d\mu_{g_i(t)}.
\]
We choose the surgery parameters sufficiently precise so that the sum
of the possible negative jumps of \(\overline\lambda\) is bounded by
\(\delta_i\). Since
\[
        \overline\lambda[g_i(t)]
        \leq\overline\lambda[h],
\]
summing the preceding inequality over the regular time intervals gives
\begin{equation}
\label{eq:late-entropy-budget}
        \int_0^\infty
        V_i(t)^{2/3}
        \int_{M_i(t)}
        |\mathcal S_i(t)|_{g_i(t)}^2
        e^{-f_i(t)}\,d\mu_{g_i(t)}\,dt
        \leq C_h\delta_i.
\end{equation}

\medskip
\noindent
We now choose the required late regular time. By Perelman's long-time
analysis, we may choose \(A_i\to\infty\) so that the thick--thin
compactness conclusions used below hold throughout the time interval
\([A_i,2A_i]\). From \eqref{eq:late-scalar-budget} and
\eqref{eq:late-entropy-budget},
\[
\begin{aligned}
        \int_{A_i}^{2A_i}
        \bigg[
        \int_{M_i(t)}Q_i(t)\,d\mu_{g_i(t)}
        &+
        V_i(t)^{2/3}
        \int_{M_i(t)}
        |\mathcal S_i(t)|_{g_i(t)}^2
        e^{-f_i(t)}\,d\mu_{g_i(t)}
        \bigg]dt
        \leq C_h\delta_i.
\end{aligned}
\]
The surgery times are discrete and hence have zero Lebesgue measure.
Applying Lemma~\ref{lem:averaging-lemma}, we may therefore choose a
regular time
\[
        t_i\in[A_i,2A_i]
\]
such that
\begin{equation}
\label{eq:late-good-scalar}
        \int_{M_i(t_i)}
        Q_i(t_i)\,d\mu_{g_i(t_i)}
        \leq
        \frac{C_h\delta_i}{A_i}
        \longrightarrow0,
\end{equation}
and
\begin{equation}
\label{eq:late-good-entropy}
\begin{aligned}
        V_i(t_i)^{2/3}
        \int_{M_i(t_i)}
        |\mathcal S_i(t_i)|_{g_i(t_i)}^2
        e^{-f_i(t_i)}\,d\mu_{g_i(t_i)}
        \leq
        \frac{C_h\delta_i}{A_i}
        \longrightarrow0.
\end{aligned}
\end{equation}
In particular, \(t_i\to\infty\). Now set
\[
        \widehat M_i:=M_i(t_i),
        \qquad
        \widehat g_i:=g_i(t_i),
        \qquad
        \widehat f_i:=f_i(t_i),
        \qquad
        \widehat\lambda_i:=\lambda[\widehat g_i].
\]
The volume bounds give
\[
        V_h
        \leq
        \operatorname{Vol}_{\widehat g_i}(\widehat M_i)
        \leq
        V_h+\delta_i,
\]
and hence
\[
        \operatorname{Vol}_{\widehat g_i}(\widehat M_i)
        \longrightarrow V_h.
\]
Moreover, entropy monotonicity, together with the controlled surgery
errors, yields
\[
        0\leq
        \overline\lambda[h]-\overline\lambda[\widehat g_i]
        \leq C_h\delta_i.
\]
Therefore
\[
        \overline\lambda[\widehat g_i]
        \longrightarrow
        -6V_h^{2/3},
        \qquad
        \widehat\lambda_i\longrightarrow-6.
\]
Let
\[
        \widehat u_i:=e^{-\widehat f_i/2}.
\]
Then
\[
        -4\Delta_{\widehat g_i}\widehat u_i
        +
        R(\widehat g_i)\widehat u_i
        =
        \widehat\lambda_i\widehat u_i,
        \qquad
        \int_{\widehat M_i}
        \widehat u_i^2\,d\mu_{\widehat g_i}=1.
\]
Testing the eigenvalue equation against \(\widehat u_i\) gives
\[
\begin{aligned}
        \widehat\lambda_i+6
        &=
        4\int_{\widehat M_i}
        |\nabla\widehat u_i|_{\widehat g_i}^2
        \,d\mu_{\widehat g_i}
        +
        \int_{\widehat M_i}
        \bigl(R(\widehat g_i)+6\bigr)
        \widehat u_i^2\,d\mu_{\widehat g_i}.
\end{aligned}
\]
Since both terms on the right-hand side are nonnegative and
\(\widehat\lambda_i\to-6\), we obtain
\begin{equation}
\label{eq:eigenfunction-energy}
        \int_{\widehat M_i}
        |\nabla\widehat u_i|_{\widehat g_i}^2
        \,d\mu_{\widehat g_i}
        \longrightarrow0,
\end{equation}
as well as
\[
        \int_{\widehat M_i}
        \bigl(R(\widehat g_i)+6\bigr)
        \widehat u_i^2\,d\mu_{\widehat g_i}
        \longrightarrow0.
\]

\medskip
\noindent
We now pass to the long-time thick part. Define the curvature scale of
\((\widehat M_i,\widehat g_i)\) by
\[
        \rho_i(x)
        :=
        \sup\left\{
        r\in(0,1]:
        \sec_{\widehat g_i}\geq-r^{-2}
        \text{ on }B_{\widehat g_i}(x,r)
        \right\}.
\]
For \(w>0\), set
\[
        \widehat M_i^+(w)
        :=
        \left\{
        x\in\widehat M_i:
        \operatorname{Vol}_{\widehat g_i}
        B_{\widehat g_i}(x,\rho_i(x))
        \geq
        w\rho_i(x)^3
        \right\},
\]
and
\[
        \widehat M_i^-(w)
        :=
        \widehat M_i\setminus\widehat M_i^+(w).
\]
Since \(t_i\to\infty\), Perelman's long-time thick-part estimates apply.
For every fixed \(w>0\), there exist constants
\[
        r_0(w)>0,
        \qquad
        \iota_0(w)>0,
        \qquad
        C_m(w)<\infty,
        \quad m=0,1,2,\ldots,
\]
such that, for every \(x_i\in\widehat M_i^+(w)\),
\[
        \operatorname{inj}_{\widehat g_i}(x_i)
        \geq\iota_0(w),
\]
and
\[
        \sup_{B_{\widehat g_i}(x_i,r_0(w))}
        |\nabla^m\operatorname{Rm}_{\widehat g_i}|_{\widehat g_i}
        \leq C_m(w)
\]
for every \(m\geq0\). The zeroth-order estimate is the thick-part
curvature estimate in Perelman's long-time analysis. Shi's local
derivative estimates give the higher-order bounds on a smaller
concentric ball, while the injectivity-radius estimate follows from
the curvature bound and the noncollapsing condition at the curvature
scale. We do not describe the details here and refer the reader to the exposition of \cite{cao2006complete}.

\noindent Now Hamilton compactness therefore applies. If
\(x_i\in\widehat M_i^+(w)\), then, after passing to a subsequence, there
exist a complete pointed smooth limit
\[
        (X,g_X,x_\infty)
\]
and, for every compact set \(K\Subset X\), embeddings
\[
        \Psi_i^K:K\longrightarrow\widehat M_i
\]
such that
\[
        (\Psi_i^K)^*\widehat g_i
        \longrightarrow g_X
\]
smoothly on \(K\). Equivalently, for every \(m\geq0\),
\[
        \|(\Psi_i^K)^*\widehat g_i-g_X\|_{C^m(K,g_X)}
        \longrightarrow0.
\]
\noindent Perelman's long-time existence theorem identifies every such
regular thick limit as a complete finite-volume hyperbolic manifold.
For completeness, the selected entropy estimates are consistent with
the same identification. Indeed, on every compact thick-limit
coordinate domain on which the pulled-back entropy minimizers have a
nonzero limit, \eqref{eq:eigenfunction-energy}, the local Harnack
inequality \cite{yau1994harnack}, and the interior Schauder estimates give smooth convergence
of the normalized minimizers to a positive constant. Pulling back
\[
        \widehat{\mathcal S}_i
        =
        \operatorname{Ric}_{\widehat g_i}
        +
        \nabla^2_{\widehat g_i}\widehat f_i
        -
        \frac{\widehat\lambda_i}{3}\widehat g_i
\]
and using \eqref{eq:late-good-entropy}, one obtains
\[
        \operatorname{Ric}_{g_X}+2g_X=0.
\]
Since \(g_X\) is smooth, this identity holds smoothly. In dimension
three the Einstein condition implies constant sectional curvature
\(-1\).
\noindent
It remains to pass from pointed thick convergence to convergence on
large domains and to prove that the complementary regions have
vanishing volume. Perelman's geometrization compactness theorem gives
a complete finite-volume hyperbolic limit \((H,h_H)\), an exhaustion
by compact smooth domains
\[
        K^j\Subset H,
        \qquad
        K^j\nearrow H,
\]
and diffeomorphisms onto their images
\[
        \psi_i^j:K^j\longrightarrow U_i^j\subset\widehat M_i
\]
such that, for every fixed \(j\),
\[
        (\psi_i^j)^*\widehat g_i
        \longrightarrow h_H
\]
smoothly on \(K^j\). In particular,
\[
        \operatorname{Vol}_{\widehat g_i}(U_i^j)
        =
        \operatorname{Vol}_{(\psi_i^j)^*\widehat g_i}(K^j)
        \longrightarrow
        \operatorname{Vol}_{h_H}(K^j).
\]

\noindent The distinguished component \(\widehat M_i\) is diffeomorphic to the
fixed closed hyperbolic manifold \(M\). Its geometrization therefore
contains precisely the hyperbolic piece determined by \(M\); equivalently,
the total hyperbolic volume is
\[
        \operatorname{Vol}_{h_H}(H)
       =
        \operatorname{Vol}_h(M)
        =
        V_h.
\]
The topological identification of the hyperbolic piece with \(M\),
followed by Mostow rigidity, identifies
\[
        (H,h_H)
        \cong
        (M,h).
\]
We henceforth make this identification and choose the exhaustion so
that
\[
        K^j\Subset M,
        \qquad
        K^j\nearrow M,
        \qquad
        \operatorname{Vol}_h(M\setminus K^j)\longrightarrow0.
\]
For instance, one may take
\[
        K^j=M\setminus B_h(p,r_j),
        \qquad
        r_j\downarrow0.
\]

\noindent For each fixed \(j\),
\[
        (\psi_i^j)^*\widehat g_i
        \longrightarrow h
\]
smoothly on \(K^j\). Consequently,
\[
        \|(\psi_i^j)^*\widehat g_i-h\|_{C^0(K^j,h)}
        \longrightarrow0
\]
and
\[
        \operatorname{Vol}_{\widehat g_i}(U_i^j)
        \longrightarrow
        \operatorname{Vol}_h(K^j).
\]
Choose a diagonal sequence \(j=j(i)\to\infty\) sufficiently slowly that
\[
        \|(\psi_i^{j(i)})^*\widehat g_i-h
        \|_{C^0(K^{j(i)},h)}
        \longrightarrow0
\]
and
\[
        \left|
        \operatorname{Vol}_{\widehat g_i}(U_i^{j(i)})
        -
        \operatorname{Vol}_h(K^{j(i)})
        \right|
        \longrightarrow0.
\]
Set
\[
        K_i^t:=K^{j(i)},
        \qquad
        G_i^t:=U_i^{j(i)},
        \qquad
        Z_i^t:=\widehat M_i\setminus G_i^t,
        \qquad
        \psi_i^t:=\psi_i^{j(i)}.
\]
Then
\[
        \|(\psi_i^t)^*\widehat g_i-h\|_{C^0(K_i^t,h)}
        \longrightarrow0
\]
and
\[
        \operatorname{Vol}_h(M\setminus K_i^t)
        \longrightarrow0.
\]

\noindent
We finally verify directly that the bad region has vanishing
\(\widehat g_i\)-volume. Since
\[
        Z_i^t=\widehat M_i\setminus G_i^t,
\]
we have
\[
\begin{aligned}
         \operatorname{Vol}_{\widehat g_i}(Z_i^t)
        &=
        |\operatorname{Vol}_{\widehat g_i}(M)
        -
        \operatorname{Vol}_{\widehat g_i}(G_i^t)|         =|\operatorname{Vol}_{\widehat g_i}(M)-\operatorname{Vol}_h(M)+\operatorname{Vol}_h(K_i^t)-\operatorname{Vol}_{\widehat g_i}(G_i^t)+\operatorname{Vol}_h(M)-\operatorname{Vol}_h(K_i^t)|\\
        &\leq
        \left|
        \operatorname{Vol}_{\widehat g_i}(M)-\operatorname{Vol}_h(M)
        \right|                                          +
        \left|
        \operatorname{Vol}_{\widehat g_i}(G_i^t)
        -
        \operatorname{Vol}_h(K_i^t)
        \right|                                          +
        \operatorname{Vol}_h(M\setminus K_i^t).
\end{aligned}
\]
The first term tends to zero by the volume bounds, the second by the
diagonal smooth convergence, and the third by the exhaustion. Hence
\[
        \operatorname{Vol}_{\widehat g_i}(Z_i^t)
        \longrightarrow0.
\]

\noindent If
\[
        \Phi_i^t:=(\psi_i^t)^{-1}
        :
        G_i^t\longrightarrow K_i^t,
\]
then the preceding \(C^0\)-convergence is equivalently written as
\[
        \sup_{x\in G_i^t}
        \left|
        \widehat g_i-(\Phi_i^t)^*h
        \right|_{(\Phi_i^t)^*h}(x)
        \longrightarrow0.
\]
Finally, \eqref{eq:late-good-scalar} and the nonnegativity of
\(R(\widehat g_i)+6\) give
\[
\begin{aligned}
        0
        &\leq
        \int_{G_i^t}
        \bigl(R(\widehat g_i)+6\bigr)
        \,d\mu_{\widehat g_i}  \\
        &\leq
        \int_{\widehat M_i}
        \bigl(R(\widehat g_i)+6\bigr)
        \,d\mu_{\widehat g_i}
        \longrightarrow0.
\end{aligned}
\]
Since \(\widehat g_i=g_i(t_i)\), all the asserted conclusions follow.
\end{proof}

\begin{remark}
 No assertion is made about the distance functions \(d_{g_i(t_i)}\).  Thus
short-circuit regions contained in \(Z_i(t_i)\) do not affect the conclusion.
The result is a tensorial \(C^0\)-convergence statement on the good part,
not a Gromov--Hausdorff statement.   
\end{remark}

\section{Pulling Back to the Initial Slice}
\noindent In the previous section, we proved the existence of a time $t=t_{i}\in [A_{i}/2,A_{i}]$ good slice for every $A_{i}\to\infty$ by using the fact that the Ricci flow with surgery on a hyperbolizable manifold is globally well-posed. In the present work, by the volume constraint and the initial condition, the volume loss at each surgery is controlled. Proving the $C^{0}$ convergence of the good part of a good time slice (adapted to the sequence of chosen initial conditions with the volume budget) to that of the final slice the hyperbolic manifold is relatively easy and follows from the standard entropy control and the compactness theory of the Ricci flow. However, the main difficulty is to pull back the convergence to the initial slice $(M,g_{i}(0))$ which what we want to prove ultimately. As discussed in the introduction, we face a significant difficulty in this part regarding the control of the appropriate norm of the metric deformation $\partial_{t}g$ along the Ricci flow. The first ingredient that we need is the so called ``Łojasiewicz" inequality (see \cite{haslhofer2012perelman} for example) in the study of gradient flow. The crucial idea here is that we need to use the gradient flow structure of the normalized Ricci flow on the space of metrics modulo diffeomorphisms and scalings. The following proposition accomplishes it  
\begin{proposition}
\label{LSpath}
Let $(\widetilde M,h)$ be a closed hyperbolic three-manifold normalized by
\[
        \text{Ric}_h=-2h,
        \qquad
        R(h)=-6.
\]
Fix \(s>7/2\). There exist a sufficiently small
\(H^s\)-neighborhood
$\mathcal U
        \subset \operatorname{Met}^{s}(\widetilde M)$
of the hyperbolic orbit
$\operatorname{Diff}^{s+1}(\widetilde M)\cdot h,$
constants
$C_{\mathrm{LS}}<\infty, 
        \theta\in\left(0,\frac12\right],$
and a constant $C<\infty$, depending only on
\((\widetilde M,h)\), \(s\), and \(\mathcal U\), with the following
property. Let \(g(t)\), \(t\in[a,b]\), be a smooth solution of the normalized
Ricci flow
\[
        \partial_tg(t)
        =
        -2\bigl(\text{Ric}_{g(t)}+2g(t)\bigr)
\]
on \(\widetilde M\), and assume that
$g(t)\in\mathcal U 
        \text{for every }t\in[a,b]$.
Let \(f(t)\) be the minimizer of Perelman's \(\lambda\)-functional,
normalized by
\[
        \int_{\widetilde M}
        e^{-f(t)}\,d\mu_{g(t)}=1,
\]
and set
\[
        \lambda(t):=\lambda[g(t)],
\mathcal S(t)
        :=
        \text{Ric}_{g(t)}
        +
        \nabla^2_{g(t)}f(t)
        -
        \frac{\lambda(t)}{3}g(t).
\]
Define the scale-invariant entropy
$\overline\lambda[g]        :=
        \lambda[g]\,
        \text{Vol}_g(\widetilde M)^{2/3},$
the entropy deficit
\[
        \mathcal D(t)
        :=
        \overline\lambda[h]
        -
        \overline\lambda[g(t)],
\]
and the scale-invariant weighted \(L^2\)-norm of the soliton defect
\[
        \mathcal N(t)
        :=
        \text{Vol}_{g(t)}(\widetilde M)^{1/3}
        \left(
        \int_{\widetilde M}
        |\mathcal S(t)|_{g(t)}^2
        e^{-f(t)}\,d\mu_{g(t)}
        \right)^{1/2}.
\]
Then
\begin{equation}
\label{eq:LS-good-component}
        \mathcal D(t)^{1-\theta}
        \leq
        C_{\mathrm{LS}}\mathcal N(t)
\end{equation}
for every \(t\in[a,b]\). Moreover,
\begin{equation}
\label{eq:LS-interval-path}
        \int_a^b\mathcal N(t)\,dt
        \leq
        \frac{C_{\mathrm{LS}}}{2\theta}
        \left(
        \mathcal D(a)^\theta
        -
        \mathcal D(b)^\theta
        \right)
        \leq
        \frac{C_{\mathrm{LS}}}{2\theta}
        \mathcal D(a)^\theta .
\end{equation}
Suppose, in addition, that
$R(g(t))+6\geq0  
        \text{on }\widetilde M
$
for every \(t\in[a,b]\). Let \(\varphi_t\) solve
\[
        \frac{d}{dt}\varphi_t
        =
        -\nabla_{g(t)}f(t)\circ\varphi_t,
        \qquad
        \varphi_a=\text{Id},
\]
and define $\widehat g(t):=\varphi_t^*g(t)$. Then
\begin{equation}
\label{eq:modified-Ricci-flow-evolution}
        \partial_t\widehat g(t)
        =
        -2\varphi_t^*
        \left(
        \mathcal S(t)
        +
        \frac{\lambda(t)+6}{3}g(t)
        \right).
\end{equation}
Consequently,
\begin{align}
        \int_a^b
        \|\partial_t\widehat g(t)\|
        _{L^2(\widetilde M,\widehat g(t))}
        \,dt
        &\leq
        C\mathcal D(a)^\theta
        +
        C
        \int_a^b
        \int_{\widetilde M}
        \bigl(R(g(t))+6\bigr)
        \,d\mu_{g(t)}\,dt.
\label{eq:modified-flow-L1L2}
\end{align}
\end{proposition}

\begin{remark}
\label{rem:L2-stronger-than-L1}
The principal conclusion of Proposition~\ref{LSpath} is the
\(L_t^1L_x^2\)-estimate
\eqref{eq:modified-flow-L1L2}. The \(L_t^1L_x^2\)-estimate implies the
\(L_t^1L_x^1\)-estimate due to the uniform boundedness of the volume
\end{remark}

\begin{remark}
\label{rem:LS-good-component-interpretation}
We need to take caution about the following fact. In the application of Proposition~\ref{LSpath}, the manifold
\(\widetilde M\) does not denote the entire, possibly singular or
surgically modified, Ricci-flow time slice. It denotes the fixed closed
hyperbolic model used to parametrize the backward-saturated good
component obtained from the intermediate comparison proposition \ref{prop:entropy-good-time-c0-hyperbolic-borel}.

\noindent More precisely, after pulling the terminal good region backward along
the surviving spacetime tube and identifying the resulting slices with
a fixed reference manifold (which we will do in the next proposition \ref{prop:flowback-c0-comparison-map}), the corresponding pulled-back metrics are
regarded as metrics on \(\widetilde M\). All quantities appearing in
Proposition~\ref{LSpath}, including
$f(t),    \lambda[g(t)],
        \mathcal D(t),
        \mathcal N(t),~
        \text{and}
        ~\mathcal U,
$
refer exclusively to this identified good component.

\noindent The proposition is applied only on those connected time intervals for
which the pulled-back good metric belongs to the hyperbolic
neighborhood \(\mathcal U\). No Lojasiewicz--Simon estimate is imposed
on the uncontrolled complement of the good component or on the full
Ricci-flow slice since such an application does not make sense. On the complementary time intervals, the duration is
controlled separately by the entropy-dissipation estimate and the
uniform gradient gap away from \(\mathcal U\).
\end{remark}

\begin{proof}
Throughout the proof, all geometric quantities are defined on the fixed
closed manifold \(\widetilde M\) parametrizing the good component.  We write
\[
        V(t):=\text{Vol}_{g(t)}(\widetilde M),
        \qquad
        \lambda(t):=\lambda[g(t)],
\]
and
\[
        d\nu_t:=e^{-f(t)}\,d\mu_{g(t)},
        \qquad
        \int_{\widetilde M}d\nu_t=1.
\]
Perelman's monotonicity formula for the scale-invariant entropy gives
\begin{eqnarray}
        \frac{d}{dt}\overline\lambda[g(t)]
        &=
        2V(t)^{2/3}
        \int_{\widetilde M}
        |\mathcal S(t)|_{g(t)}^2
        e^{-f(t)}\,d\mu_{g(t)}
        \notag
        +
        \frac{2}{3}V(t)^{2/3}\lambda(t)
        \bigl(\lambda(t)-\overline R(t)\bigr),
\label{eq:LS-proof-monotonicity}
\end{eqnarray}
where
\[
        \overline R(t)
        :=
        \frac{1}{V(t)}
        \int_{\widetilde M}R(g(t))\,d\mu_{g(t)}.
\]
Since \(\widetilde M\) is of negative Yamabe type,
\[
        \lambda(t)\leq0,
\]
and the constant function is an admissible test function in the definition
of \(\lambda\), so that
\[
        \lambda(t)\leq\overline R(t).
\]
The second term on the right-hand side of
\eqref{eq:LS-proof-monotonicity} is therefore nonnegative. Recalling the
definition of \(\mathcal N(t)\), we obtain
\begin{equation}
\label{eq:LS-proof-dissipation}
        -\mathcal D'(t)
        =
        \frac{d}{dt}\overline\lambda[g(t)]
        \geq
        2\mathcal N(t)^2.
\end{equation}
We next place the Lojasiewicz--Simon inequality in the precise Hilbert
space framework used here. Fix
$s>\frac{7}{2}$.
Let \(\mathcal M^s(\widetilde M)\) denote the Hilbert manifold of
\(H^s\)-metrics on \(\widetilde M\). We work modulo the action
\[
        (\varphi,c)\cdot g:=c\,\varphi^*g,
        \qquad
        (\varphi,c)\in
        \text{Diff}^{s+1}(\widetilde M)\times\mathbb R_+.
\]
Let \(\delta_h\) denote the divergence with respect to \(h\). A local
slice transverse to the diffeomorphism and scaling orbit of \(h\) is
given by
\[
\mathscr S_h^s(\varepsilon)
:=
\left\{
\begin{array}{l|l}
h+\ell
&
\begin{array}{l}
\ell\in H^s(S^2T^*\widetilde M),\quad
\delta_h\ell=0,\\[1mm]
\displaystyle
\int_{\widetilde M}\tr_h\ell\,d\mu_h=0,\quad
\|\ell\|_{H^s(h)}<\varepsilon,\quad
h+\ell>0
\end{array}
\end{array}
\right\}.
\]
Accordingly, we define
\[
\mathcal U
:=
\left\{
        g\in\mathcal M^s(\widetilde M):
        c\,\varphi^*g\in\mathscr S_h^s(\varepsilon)
        \text{ for some }
        \varphi\in\text{Diff}^{s+1}(\widetilde M),\ c>0
\right\}.
\]
After decreasing \(\varepsilon\), the functional
\[
        g\longmapsto
        \overline\lambda[g]
        =
        \lambda[g]\text{Vol}_g(\widetilde M)^{2/3}
\]
is real analytic on the slice, and the Lojasiewicz--Simon inequality at
the hyperbolic critical metric gives constants
$C_{\mathrm{LS}}<\infty,    
        \theta\in\left(0,\frac12\right],
$
such that
\begin{equation}
\label{eq:abstract-LS}
        \bigl(
        \overline\lambda[h]-\overline\lambda[g]
        \bigr)^{1-\theta}
        \leq
        C_{\mathrm{LS}}
        \left\|
        \nabla_{\mathscr S}\overline\lambda(g)
        \right\|_{\mathcal H_g}
\end{equation}
for every \(g\in\mathcal U\) (see
\cite{kroncke2013stability} for details. Here analyticity is of extreme importance). Here \(\nabla_{\mathscr S}\) denotes the
gradient restricted to the local slice and
\[
        (A,B)_{\mathcal H_g}
        :=
        V^{2/3}
        \int_{\widetilde M}
        \langle A,B\rangle_g\,d\nu
\]
is the scale-invariant weighted \(L^2\)-inner product.
We compare the gradient appearing in \eqref{eq:abstract-LS} with the
soliton defect. For a variation
\[
        a=\left.\frac{d}{d\sigma}\right|_{\sigma=0}g_\sigma,
\]
the first-variation formulas for \(\lambda\) and the volume give after straightforward computations
\begin{align}
        D\overline\lambda_g(a)
        &=
        -V^{2/3}
        \int_{\widetilde M}
        \left\langle
        \text{Ric}_g+\nabla_g^2f,a
        \right\rangle_g\,d\nu
        \notag\\
        &\quad
        +
        \frac{\lambda[g]}{3}V^{-1/3}
        \int_{\widetilde M}
        \tr_g a\,d\mu_g.
\label{eq:first-variation-scaled-lambda}
\end{align}
Since \(d\mu_g=e^f\,d\nu\), this may be written as
\[
        D\overline\lambda_g(a)
        =
        -V^{2/3}
        \int_{\widetilde M}
        \left\langle
        \mathfrak G_g,a
        \right\rangle_g\,d\nu,
\]
where
\[
        \mathfrak G_g
        :=
        \text{Ric}_g+\nabla_g^2f
        -
        \frac{\lambda[g]}{3V}e^f g.
\]
Thus \(-\mathfrak G_g\) is the ambient gradient of
\(\overline\lambda\) with respect to \(\mathcal H_g\). Since the slice
gradient is its orthogonal projection,
\begin{equation}
\label{eq:slice-gradient-projection}
        \left\|
        \nabla_{\mathscr S}\overline\lambda(g)
        \right\|_{\mathcal H_g}
        \leq
        \|\mathfrak G_g\|_{\mathcal H_g}.
\end{equation}
Recall that
\[
        \mathcal S_g
        =
        \text{Ric}_g+\nabla_g^2f
        -
        \frac{\lambda[g]}{3}g.
\]
Consequently,
\begin{equation}
\label{eq:gradient-defect-difference}
        \mathfrak G_g-\mathcal S_g
        =
        -\frac{\lambda[g]}{3}
        \left(
        \frac{e^f}{V}-1
        \right)g.
\end{equation}
We show that the difference on the right is controlled by
\(\mathcal S_g\). First, set
\[
        \phi:=f-\log V,
        \qquad
        d\overline\mu:=V^{-1}d\mu_g.
\]
Then
\[
        d\nu=e^{-\phi}\,d\overline\mu,
        \qquad
        \int_{\widetilde M}e^{-\phi}\,d\overline\mu=1.
\]
The Euler--Lagrange equation for the entropy minimizer is
\[
        2\Delta_gf-|\nabla f|_g^2+R(g)=\lambda[g].
\]
Taking the trace of \(\mathcal S_g\) and using this equation gives
\begin{equation}
\label{eq:phi-equation}
        \Delta_g\phi
        =
        |\nabla\phi|_g^2
        -
        \tr_g\mathcal S_g.
\end{equation}
Define the following
\[
        \overline\phi
        :=
        \int_{\widetilde M}\phi\,d\overline\mu,
        \qquad
        \phi_0:=\phi-\overline\phi.
\]
Because \(g\in\mathcal U\), the metrics are uniformly equivalent to
\(h\), their elliptic constants are uniformly controlled, and the first
positive eigenvalue of \(-\Delta_g\) is bounded uniformly away from zero.
It follows that
\[
        \|\phi_0\|_{H^2(g)}
        \leq
        C\|\Delta_g\phi\|_{L^2(g)}.
\]
Using \eqref{eq:phi-equation}, we obtain
\begin{eqnarray}
        \|\phi_0\|_{H^2(g)}
        &\leq
        C\|\mathcal S_g\|_{L^2(g)}
        +
        C\bigl\||\nabla\phi|_g^2\bigr\|_{L^2(g)}
        \notag\leq
        C\|\mathcal S_g\|_{L^2(g)}
        +
        C\|\nabla\phi\|_{L^\infty(g)}
        \|\nabla\phi_0\|_{L^2(g)}.
\label{eq:phi-H2-preabsorption}
\end{eqnarray}
The entropy minimizer depends smoothly on the metric in the
\(H^s\)-neighborhood \(\mathcal U\), and at \(h\) it is constant.
Shrinking \(\mathcal U\), if necessary, we may therefore assume that
$\|\nabla\phi\|_{L^\infty(g)}$
is sufficiently small. The final term in
\eqref{eq:phi-H2-preabsorption} can then be absorbed, giving
\begin{equation}
\label{eq:phi-zero-mode-control}
        \|\phi-\overline\phi\|_{H^2(g)}
        \leq
        C\|\mathcal S_g\|_{L^2(g)}.
\end{equation}
It remains to estimate the constant mode. Since
\[
        \int_{\widetilde M}e^{-\phi}\,d\overline\mu=1
\]
and \(\phi\) is uniformly small in \(C^0\), Taylor's formula gives for small $\varphi$
\[
\begin{aligned}
        0
        &=
        \int_{\widetilde M}
        \bigl(e^{-\phi}-1\bigr)\,d\overline\mu=
        -\overline\phi
        +
        O\left(
        \|\phi\|_{L^2(d\overline\mu)}^2
        \right).
\end{aligned}
\]
Hence
\[
        |\overline\phi|
        \leq
        C\|\phi\|_{L^2(d\overline\mu)}^2.
\]
Combining this estimate with
\eqref{eq:phi-zero-mode-control}, and shrinking \(\mathcal U\) once
more, yields
\begin{equation}
\label{eq:full-phi-control}
        \|\phi\|_{L^2(d\overline\mu)}
        \leq
        C\|\mathcal S_g\|_{L^2(g)}.
\end{equation}
Since the measures
\[
        d\nu=e^{-\phi}d\overline\mu,
        \qquad
        d\overline\mu,
        \qquad
        d\mu_g
\]
are uniformly equivalent on \(\mathcal U\), and
\[
        \frac{e^f}{V}-1=e^\phi-1,
\]
equation \eqref{eq:full-phi-control} implies
\begin{equation}
\label{eq:exponential-phi-control}
        \left\|
        \frac{e^f}{V}-1
        \right\|_{L^2(d\nu)}
        \leq
        C
        \|\mathcal S_g\|_{L^2(d\nu)}.
\end{equation}
The ordinary entropy \(\lambda[g]\) and the volume \(V\) are uniformly
bounded above and below on \(\mathcal U\). Therefore
\eqref{eq:gradient-defect-difference} and
\eqref{eq:exponential-phi-control} give
\[
        \|\mathfrak G_g\|_{\mathcal H_g}
        \leq
        C V^{1/3}
        \left(
        \int_{\widetilde M}
        |\mathcal S_g|_g^2\,d\nu
        \right)^{1/2}.
\]
In the notation of the proposition,
\begin{equation}
\label{eq:ambient-gradient-controlled-N}
        \|\mathfrak G_{g(t)}\|_{\mathcal H_{g(t)}}
        \leq
        C\mathcal N(t).
\end{equation}
Combining \eqref{eq:abstract-LS},
\eqref{eq:slice-gradient-projection}, and
\eqref{eq:ambient-gradient-controlled-N}, we conclude that
\begin{equation}
\label{eq:final-LS-defect}
        \mathcal D(t)^{1-\theta}
        \leq
        C_{\mathrm{LS}}\mathcal N(t).
\end{equation}
We now derive the path estimate. At times for which
\(\mathcal N(t)>0\), equation \eqref{eq:final-LS-defect} gives
\[
        \frac{1}{\mathcal N(t)}
        \leq
        C_{\mathrm{LS}}\mathcal D(t)^{\theta-1}.
\]
Using \eqref{eq:LS-proof-dissipation}, we obtain
\[
\begin{aligned}
        \mathcal N(t)
        &=
        \frac{\mathcal N(t)^2}{\mathcal N(t)}\leq
        C_{\mathrm{LS}}
        \mathcal N(t)^2
        \mathcal D(t)^{\theta-1}\leq
        -\frac{C_{\mathrm{LS}}}{2}
        \mathcal D(t)^{\theta-1}\mathcal D'(t).
\end{aligned}
\]
The same inequality holds at the remaining times by continuity. Integration
over \([a,b]\) therefore gives
\begin{align}
        \int_a^b\mathcal N(t)\,dt
        &\leq
        \frac{C_{\mathrm{LS}}}{2\theta}
        \left(
        \mathcal D(a)^\theta
        -
        \mathcal D(b)^\theta
        \right)
        \notag
        \leq
        \frac{C_{\mathrm{LS}}}{2\theta}
        \mathcal D(a)^\theta.
\label{eq:LS-final-path-estimate}
\end{align}
\noindent We finally estimate the deformation in the entropy gauge. Let
\(\varphi_t\) solve
\[
        \dot\varphi_t
        =
        -\nabla_{g(t)}f(t)\circ\varphi_t,
        \qquad
        \varphi_a=\text{Id},
\]
and set
\[
        \widehat g(t):=\varphi_t^*g(t).
\]
Then
\begin{align}
        \partial_t\widehat g(t)
        &=
        \varphi_t^*
        \left(
        \partial_tg(t)
        +
        \mathcal L_{-\nabla f(t)}g(t)
        \right)
        \notag\\
        &=
        -2\varphi_t^*
        \left(
        \text{Ric}_{g(t)}
        +
        \nabla^2_{g(t)}f(t)
        +
        2g(t)
        \right)
        \notag\\
        &=
        -2\varphi_t^*
        \left(
        \mathcal S(t)
        +
        \frac{\lambda(t)+6}{3}g(t)
        \right).
\label{eq:entropy-gauge-evolution-proof}
\end{align}
By invariance of tensor norms and volume under pullback,
\begin{align}
        \|\partial_t\widehat g(t)\|
        _{L^2(\widetilde M,\widehat g(t))}
        &\leq
        2\|\mathcal S(t)\|_{L^2(\widetilde M,g(t))}
        +
        \frac{2}{\sqrt3}
        V(t)^{1/2}|\lambda(t)+6|.
\label{eq:gauged-L2-pointwise-time}
\end{align}
Since \(g(t)\in\mathcal U\), the measures
\(e^{-f(t)}d\mu_{g(t)}\) and \(d\mu_{g(t)}\) are uniformly equivalent,
and \(V(t)\) is uniformly bounded above and below. It follows that
\[
        \|\mathcal S(t)\|_{L^2(\widetilde M,g(t))}
        \leq
        C\mathcal N(t).
\]
Suppose now that \(R(g(t))+6\geq0\). The Rayleigh quotient defining
\(\lambda[g(t)]\) gives
\[
        \lambda(t)\geq-6.
\]
Testing the same quotient with the constant function gives
\[
\begin{aligned}
        \lambda(t)
        &\leq
        \frac{1}{V(t)}
        \int_{\widetilde M}R(g(t))\,d\mu_{g(t)}=
        -6+
        \frac{1}{V(t)}
        \int_{\widetilde M}
        \bigl(R(g(t))+6\bigr)\,d\mu_{g(t)}.
\end{aligned}
\]
Consequently,
\begin{equation}
\label{eq:lambda-plus-six-proof}
        0\leq\lambda(t)+6
        \leq
        \frac{1}{V(t)}
        \int_{\widetilde M}
        \bigl(R(g(t))+6\bigr)\,d\mu_{g(t)}.
\end{equation}
Combining \eqref{eq:gauged-L2-pointwise-time} and
\eqref{eq:lambda-plus-six-proof}, and using the uniform lower volume
bound, we obtain
\begin{align}
        \int_a^b
        \|\partial_t\widehat g(t)\|
        _{L^2(\widetilde M,\widehat g(t))}
        \,dt
        &\leq
        C\mathcal D(a)^\theta
        +
        C\int_a^b
        \int_{\widetilde M}
        \bigl(R(g(t))+6\bigr)
        \,d\mu_{g(t)}\,dt.
\end{align}
This proves the asserted \(L_t^1L_x^2\)-estimate.

\noindent Finally, the volumes of
\((\widetilde M,\widehat g(t))\) are uniformly bounded, since pullback
does not change volume and \(g(t)\in\mathcal U\). Therefore Hölder's
inequality gives
\[
        \|\partial_t\widehat g(t)\|_{L^1}
        \leq
        C
        \|\partial_t\widehat g(t)\|_{L^2}.
\]
Integration in time proves the corresponding
\(L_t^1L_x^1\)-estimate and completes the proof.
\end{proof}

\begin{remark}
\label{rem:closed-good-tube}
In the application of Proposition~\ref{LSpath}, the good spacetime tube
is not regarded merely as a family of compact domains with boundary.
After discarding the surgery regions and the components whose total
volume tends to zero, each regular time section of the surviving good
tube is closed, if necessary, across its negligible boundary components
by the standard post-surgery caps.  The surgery parameters and the
closing regions are chosen so that the total volume introduced or
discarded in this procedure is \(o(1)\).

\noindent The resulting time sections are identified by the surviving spacetime
diffeomorphisms, with a fixed closed manifold
$\widetilde M$.
Thus the metrics on the good tube may be viewed as a family
$\widetilde g_i(t), 0\leq t\leq t_i,$
of smooth metrics on \(\widetilde M\).  All applications of the
Lojasiewicz--Simon inequality are made to this closed family and not to
the uncontrolled complement of the good tube in the full Ricci-flow
slice.

\noindent In particular, the notation
\[
        \widetilde f_i(t),\qquad
        \widetilde\lambda_i(t),\qquad
        \widetilde{\mathcal S}_i(t),\qquad
        \widetilde{\mathcal D}_i(t)
\]
denotes respectively the Perelman minimizer, the entropy, the soliton
defect, and the entropy deficit of
\((\widetilde M,\widetilde g_i(t))\).  Since the closing regions have
vanishing volume and lie outside the inner buffered good region, they do
not affect the limiting metric comparison or the volume-exhaustion
conclusion as we will do in the final proposition \ref{prop:flowback-c0-comparison-map}.
\end{remark}

\begin{remark}
\label{rem:LS-only-good-tube}
Proposition~\ref{LSpath} is used only on those time intervals for which
the metric \(\widetilde g_i(t)\) on the closed good component lies in the
fixed Sobolev neighborhood \(\mathcal U\) of the hyperbolic orbit.
Accordingly, the entropy deficit and the defect norm appearing in the
Lojasiewicz--Simon inequality are
\[
        \widetilde{\mathcal D}_i(t)
        :=
        \overline\lambda[h]
        -
        \overline\lambda[\widetilde g_i(t)]
\]
and
\[
        \widetilde{\mathcal N}_i(t)
        :=
        \text{Vol}_{\widetilde g_i(t)}(\widetilde M)^{1/3}
        \left(
        \int_{\widetilde M}
        |\widetilde{\mathcal S}_i(t)|^2
        e^{-\widetilde f_i(t)}
        \,d\mu_{\widetilde g_i(t)}
        \right)^{1/2}.
\]
Thus both quantities are defined on the same closed good component, and
Proposition~\ref{LSpath} gives
\[
        \widetilde{\mathcal D}_i(t)^{1-\theta}
        \leq
        C_{\mathrm{LS}}\widetilde{\mathcal N}_i(t)
\]
whenever
\[
        \widetilde g_i(t)\in\mathcal U.
\]

\noindent No Lojasiewicz--Simon estimate is asserted for the full surgically
modified time slice or for the discarded thin and high-curvature
regions.  On the complementary set of times, when
\(\widetilde g_i(t)\notin\mathcal U\), the duration of that time set is
controlled by the entropy dissipation and the uniform gradient gap away
from \(\mathcal U\).
\end{remark}

\begin{remark}
\label{rem:entropy-gauge-good-tube}
Let \(\varphi_i(t)\) be generated on the closed good component by the
negative gradient of its Perelman minimizer,
\[
        \frac{d}{dt}\varphi_i(t)
        =
        -\nabla_{\widetilde g_i(t)}
        \widetilde f_i(t)\circ\varphi_i(t),
        \qquad
        \varphi_i(0)=\text{Id},
\]
and set
\[
        \widehat g_i(t)
        :=
        \varphi_i(t)^*\widetilde g_i(t).
\]
Then the diffeomorphism-modified flow satisfies
\[
        \partial_t\widehat g_i(t)
        =
        -2\varphi_i(t)^*
        \left(
        \text{Ric}_{\widetilde g_i(t)}
        +
        \nabla^2_{\widetilde g_i(t)}\widetilde f_i(t)
        +
        2\widetilde g_i(t)
        \right).
\]
Writing
\[
        \widetilde{\mathcal S}_i(t)
        :=
        \text{Ric}_{\widetilde g_i(t)}
        +
        \nabla^2_{\widetilde g_i(t)}\widetilde f_i(t)
        -
        \frac{\widetilde\lambda_i(t)}{3}
        \widetilde g_i(t),
\]
the deformation tensor is therefore
\[
        \widetilde{\mathcal E}_i(t)
        :=
        \text{Ric}_{\widetilde g_i(t)}
        +
        \nabla^2_{\widetilde g_i(t)}\widetilde f_i(t)
        +
        2\widetilde g_i(t)
        =
        \widetilde{\mathcal S}_i(t)
        +
        \frac{\widetilde\lambda_i(t)+6}{3}
        \widetilde g_i(t).
\]
Consequently, Proposition~\ref{LSpath}, together with the scalar-defect
estimate, yields
\[
        \int_0^{t_i}
        \|\widetilde{\mathcal E}_i(t)\|
        _{L^2(\widetilde M,\widetilde g_i(t))}
        \,dt
        \longrightarrow0.
\]
This is the path-length estimate used to pull the terminal hyperbolic
comparison back to the initial slice.  In particular, no separate
estimate for
\(\nabla^2\widetilde f_i\)
is required: the Hessian term is incorporated into the exact
diffeomorphism-modified deformation tensor. This is important.
\end{remark}

\begin{remark}
\label{rem:closing-volume-error}
All sets introduced in closing the surviving good component, as well as
the regions removed at surgery times, are included in the exceptional
sets.  Their total volume is absorbed into an error
\[
        \varepsilon_i^{\mathrm{cl}}\longrightarrow0.
\]
Accordingly, replacing a time section of the good tube by its closed
model changes the relevant volume estimates only by
\(o(1)\).  The path-length and metric-comparison arguments are carried
out on a smaller region disjoint from the closing collars, so
the final tensorial \(C^0\)-comparison on the good set is unchanged.
\end{remark}

\noindent Next we prove a crucial lemma regarding the gradient gap used in the final proposition \ref{prop:flowback-c0-comparison-map}. This proposition is vital in the sense that it helps to quantify the sets in the time interval $[0,t_{i}]$ ($t_{i}\to\infty$, and there are finitely many surgery times in this interval) that have large soliton defect in the appropriate norm.

\begin{lemma}
\label{lem:uniform-soliton-defect-gap}
Let \((\widetilde M,h)\) be a closed hyperbolic three-manifold, and
fix \(s>7/2\). Denote by
\[
        \mathscr H_h
        :=
        \left\{
                c\,\phi^*h:
                c>0,\ 
                \phi\in\operatorname{Diff}^{s+1}(\widetilde M)
        \right\}
\]
the hyperbolic orbit under diffeomorphisms and constant rescalings,
and let
$\mathcal U\subset\operatorname{Met}^s(\widetilde M)$
be an open
\(\operatorname{Diff}^{s+1}(\widetilde M)\times\mathbb R_+\)-invariant
neighborhood of \(\mathscr H_h\).

\noindent Let \(\mathscr K\subset\operatorname{Met}^s(\widetilde M)\) be a
diffeomorphism-invariant class of smooth metrics satisfying the
following two properties:
\begin{enumerate}
\item there exist \(0<v_0\leq V_0<\infty\) such that
\[
        v_0
        \leq
        \operatorname{Vol}_g(\widetilde M)
        \leq
        V_0
        \qquad\text{for every }g\in\mathscr K;
\]
\item \(\mathscr K\) is relatively compact in \(H^s\) modulo
\(\operatorname{Diff}^{s+1}(\widetilde M)\): for every sequence
\(g_j\in\mathscr K\), there exist a subsequence and
\(\phi_j\in\operatorname{Diff}^{s+1}(\widetilde M)\) such that
\[
        \phi_j^*g_j\longrightarrow g_\infty
        \qquad\text{in }H^s
\]
for some \(H^s\)-metric \(g_\infty\) on \(\widetilde M\).
\end{enumerate}

\noindent For \(g\in\mathscr K\), let \(f_g\) be the normalized minimizer of
Perelman's \(\lambda\)-functional,
\[
        \int_{\widetilde M}e^{-f_g}\,d\mu_g=1,
\]
and set
\[
        S_g
        :=
        \operatorname{Ric}_g+\nabla_g^2f_g
        -\frac{\lambda[g]}{3}g,
\]
\[
        \mathcal N(g)
        :=
        \operatorname{Vol}_g(\widetilde M)^{1/3}
        \left(
        \int_{\widetilde M}
        |S_g|_g^2e^{-f_g}\,d\mu_g
        \right)^{1/2}.
\]
Then there exists a constant
$\gamma_{\mathcal U}>0$
such that
\begin{equation}
\label{eq:uniform-soliton-defect-gap}
        g\in\mathscr K\setminus\mathcal U
        \quad\Longrightarrow\quad
        \mathcal N(g)\geq\gamma_{\mathcal U}.
\end{equation}
\end{lemma}

\begin{proof}
Suppose that \eqref{eq:uniform-soliton-defect-gap} were false.
Then there would exist \(g_j\in\mathscr K\setminus\mathcal U\)
such that
\[
        \mathcal N(g_j)\longrightarrow0.
\]
By the compactness assumption, after passing to a subsequence and
pulling back by diffeomorphisms, we may assume that
\[
        g_j\longrightarrow g_\infty
        \qquad\text{in }H^s.
\]
Since \(\mathcal U\) is diffeomorphism invariant and open,
\[
        g_\infty\notin\mathcal U.
\]
Now note that the lowest eigenvalue of
\[
        -4\Delta_g+R(g)
\]
is simple. Hence standard elliptic perturbation theory, together
with \(s>7/2\), gives
\[
        \lambda[g_j]\longrightarrow\lambda[g_\infty],
        \qquad
        f_{g_j}\longrightarrow f_{g_\infty}
\]
in the Sobolev topologies required to pass to the soliton-defect
tensor. Consequently,
\[
        \mathcal N(g_\infty)=0,
\]
and therefore
\begin{equation}
\label{eq:limit-gradient-soliton}
        \operatorname{Ric}_{g_\infty}
        +
        \nabla_{g_\infty}^2f_{g_\infty}
        =
        \frac{\lambda[g_\infty]}{3}g_\infty.
\end{equation}
Elliptic regularity implies that \(g_\infty\) and \(f_{g_\infty}\)
are smooth. Now taking the trace of
\eqref{eq:limit-gradient-soliton} gives
\begin{equation}
\label{eq:trace-limit-gradient-soliton}
        R(g_\infty)
        +
        \Delta_{g_\infty}f_{g_\infty}
        =
        \lambda[g_\infty].
\end{equation}
On the other hand, the Euler--Lagrange equation for the normalized
\(\lambda\)-minimizer is
\begin{equation}
\label{eq:lambda-minimizer-limit}
        2\Delta_{g_\infty}f_{g_\infty}
        -
        |\nabla f_{g_\infty}|_{g_\infty}^2
        +
        R(g_\infty)
        =
        \lambda[g_\infty].
\end{equation}
Subtracting \eqref{eq:trace-limit-gradient-soliton} from
\eqref{eq:lambda-minimizer-limit}, we obtain
\[
        \Delta_{g_\infty}f_{g_\infty}
        =
        |\nabla f_{g_\infty}|_{g_\infty}^2.
\]
Integration over the closed manifold \(\widetilde M\) yields
\[
        \int_{\widetilde M}
        |\nabla f_{g_\infty}|_{g_\infty}^2
        \,d\mu_{g_\infty}
        =0.
\]
Thus \(f_{g_\infty}\) is constant, and
\eqref{eq:limit-gradient-soliton} reduces to
\[
        \operatorname{Ric}_{g_\infty}
        =
        \frac{\lambda[g_\infty]}{3}g_\infty.
\]
Hence \(g_\infty\) is Einstein. In dimension three it has constant
sectional curvature. Since \(\widetilde M\) is a closed hyperbolic
manifold, the curvature is negative, and Mostow rigidity gives
\[
        g_\infty=c\,\phi^*h
\]
for some \(c>0\) and
\(\phi\in\operatorname{Diff}^{s+1}(\widetilde M)\). Therefore
\(g_\infty\in\mathscr H_h\subset\mathcal U\), contradicting
\(g_\infty\notin\mathcal U\).
\end{proof}

\noindent Now we arrive at the final proposition, which will help complete the proof of the main theorem.

\begin{proposition}
\label{prop:flowback-c0-comparison-map}
Under the hypotheses of Proposition~\ref{prop:entropy-good-time-c0-hyperbolic-borel} and
Remarks~6-9, let
\[
        \bigl(\mathcal M_i,\mathfrak t,g_i(t)\bigr)
\]
be the corresponding normalized Ricci-flow spacetimes with surgery.
Let \(M_i(s)=\mathfrak t^{-1}(s)\), and let \(g_i(s)\) be the induced metric.
Assume that there exist \(A_i\to\infty\), with \(2A_i\leq T_i\), and regular
times
$t_i\in[A_i,2A_i]$.
Assume that on the good slice \(M_i(t_i)\) there exist 
$G_i^t\subset M_i(t_i),~
        Z_i^t:=M_i(t_i)\setminus G_i^t$,
smooth compact domains \(K_i^t\subset M\), and diffeomorphisms
\[
        \psi_i^t:K_i^t\longrightarrow G_i^t
\]
such that
\[
        \operatorname{Vol}_{g_i(t_i)}(Z_i^t)\leq b_i^t,
        b_i^t\to0,
\]
and
\[
        \operatorname{Vol}_{h}(M\setminus K_i^t)\to0.
\]
Assume moreover that the terminal \(C^0\)-comparison holds:
\[
        \bigl\|(\psi_i^t)^*g_i(t_i)-h\bigr\|_{C^0(K_i^t,h)}
        \leq
        \varepsilon_i^t,
        \qquad
        \varepsilon_i^t\to0.
\]
Let \(\mathcal S_i^0\subset M_i(0)\) be the set of points whose worldlines
survive until time \(t_i\), and let
\[
        P_i:=P_{0,t_i}:\mathcal S_i^0\longrightarrow M_i(t_i)
\]
be the survival map. Assume that, after enlarging \(Z_i^t\) by a set of
\(g_i(t_i)\)-volume at most \(r_i^t\to0\), one has
$G_i^t\subset P_i(\mathcal S_i^0)$.
Then there exist
$G_i^0\subset M_i(0),        
        Z_i^0:=M_i(0)\setminus G_i^0$,
subsets \(K_i^0\subset K_i^t\), and maps
$\psi_i^0:K_i^0\longrightarrow G_i^0$
which are diffeomorphisms onto their images, such that
\[
        \operatorname{Vol}_{g_i(0)}(Z_i^0)\to0,
        \qquad
        \operatorname{Vol}_{h}(M\setminus K_i^0)\to0,
\]
and
\[
        \bigl\|(\psi_i^0)^*g_i(0)-h\bigr\|_{C^0(K_i^0,h)}
        \to0,
\]
i.e.,
\[
        e^{-2\eta_i}(1-\varepsilon_i^t)h
        \leq
        (\psi_i^0)^*g_i(0)
        \leq
        e^{2\eta_i}(1+\varepsilon_i^t)h
\]
on \(K_i^0\).
\end{proposition}
\begin{proof}
Since \(A_i\to\infty\) and \(t_i\in[A_i,2A_i]\), one has
\(t_i\to\infty\). All estimates below are uniform in the length of the
interval \([0,t_i]\). This uniform estimate is the most difficult part of the proof. 

\noindent After replacing \(K_i^t\) by a slightly smaller smooth domain and
enlarging \(Z_i^t\) by a set of \(g_i(t_i)\)-volume \(o(1)\), choose
smooth compact domains
\[
        K_i^-\Subset K_i^+\Subset K_i^t
\]
such that
\[
        \operatorname{Vol}_h(K_i^t\setminus K_i^-)\to0,
\]
and set
\[
        G_i^{t,\pm}:=\psi_i^t(K_i^\pm).
\]
Enlarging the terminal bad set once more by a set of volume \(o(1)\), we
may assume that every point of \(G_i^{t,+}\) has a worldline surviving
on \([0,t_i]\), and that its backward saturation is disjoint from all
surgery regions. Define the regular buffered tubes
\[
        G_i^\pm(t):=P_{t,t_i}^{-1}(G_i^{t,\pm}),
        \qquad
        \mathcal G_i^\pm
        :=
        \bigcup_{0\leq t\leq t_i}G_i^\pm(t).
\]
The buffer between \(\mathcal G_i^-\) and the lateral boundary of
\(\mathcal G_i^+\) permits all diffeomorphism modifications to be
performed on the outer tube while retaining the exact evolution
equation on the inner tube.

\noindent We use the closed good-component construction associated with
\(\mathcal G_i^+\). Denote the corresponding closed normalized flow by
$(\widetilde M_i,\widetilde g_i(t)),
0\leq t\leq t_i$,
which agrees with \(g_i(t)\) on \(G_i^+(t)\).

\noindent Choose a smooth cutoff on \(\mathcal G_i^+\), equal to one on
\(\mathcal G_i^-\) and vanishing near the lateral boundary, and let
\(\Psi_i(t)\) be the composition of the spacetime identification with
the diffeomorphisms generated by the resulting cutoff of
\(-\nabla\widetilde f_i(t)\). Thus
\[
        \Psi_i(t):G_i^+(0)\longrightarrow G_i^+(t),
        \qquad
        \Psi_i(0)=\operatorname{Id},
\]
and on \(G_i^-(0)\),
\begin{equation}
\label{eq:modified-flow-metric-evolution}
        \partial_t\bigl(\Psi_i(t)^*g_i(t)\bigr)
        =
        -2\Psi_i(t)^*\widetilde{\mathcal D}_i(t).
\end{equation}
In particular, no separate estimate for
\(\nabla^2\widetilde f_i\) is required. This is the main reason behind using the diffeomorphism-modified flow. This construction is extremely important. 

\noindent
All entropy quantities below are computed on the closed completed flow
\[
        \bigl(\widetilde M_i,\widetilde g_i(t)\bigr),
        \qquad 0\leq t\leq t_i .
\]
Let \(\widetilde f_i(t)\) be the minimizer of Perelman's
\(\lambda\)-functional, normalized by
\[
        \int_{\widetilde M_i}
        e^{-\widetilde f_i(t)}
        \,d\mu_{\widetilde g_i(t)}
        =1,
\]
and set
\[
        \widetilde\lambda_i(t)
        :=
        \lambda[\widetilde g_i(t)],
        \qquad
        \widetilde S_i(t)
        :=
        \operatorname{Ric}_{\widetilde g_i(t)}
        +
        \nabla_{\widetilde g_i(t)}^{2}\widetilde f_i(t)
        -
        \frac{\widetilde\lambda_i(t)}{3}
        \widetilde g_i(t).
\]
Define
\[
        \widetilde e_i(t)
        :=
        \overline\lambda[h]
        -
        \overline\lambda[\widetilde g_i(t)]
\]
and
\[
        \widetilde{\mathcal N}_i(t)
        :=
        \operatorname{Vol}_{\widetilde g_i(t)}
        (\widetilde M_i)^{1/3}
        \left(
        \int_{\widetilde M_i}
        |\widetilde S_i(t)|_{\widetilde g_i(t)}^{2}
        e^{-\widetilde f_i(t)}
        \,d\mu_{\widetilde g_i(t)}
        \right)^{1/2}.
\]
On every regular time interval, entropy monotonicity gives
\begin{equation}
\label{eq:closed-entropy-dissipation}
        -\frac{d}{dt}\widetilde e_i(t)
        \geq
        c_0\widetilde{\mathcal N}_i(t)^2,
\end{equation}
where \(c_0>0\) is independent of \(i\).

\noindent Let \(\mathcal T_i\subset(0,t_i)\) denote the surgery times.  Choosing
the surgery and closing parameters sufficiently precise, we may assume
that
\[
        \rho_i
        :=
        \sum_{s\in\mathcal T_i}
        \bigl[
        \widetilde e_i(s^+)-\widetilde e_i(s^-)
        \bigr]_+
        \longrightarrow0
\]
and
\[
        \sigma_i
        :=
        \sum_{s\in\mathcal T_i}
        \bigl[
        \widetilde e_i(s^+)^\theta
        -
        \widetilde e_i(s^-)^\theta
        \bigr]_+
        \longrightarrow0.
\]
By Proposition~2.4 and the closing-volume estimate,
\[
        \varepsilon_i
        :=
        \widetilde e_i(0)+\rho_i
        \longrightarrow0.
\]
Consequently,
\begin{equation}
\label{eq:closed-entropy-budget}
        \sup_{0\leq t\leq t_i}\widetilde e_i(t)
        \leq\varepsilon_i,
        \qquad
        \int_0^{t_i}
        \widetilde{\mathcal N}_i(t)^2\,dt
        \leq
        \frac{\varepsilon_i}{c_0},
\end{equation}
where, here and below, time integrals are taken over the regular
portions of the surgical flow.

\noindent Let \(\phi_i(t)\) be generated on \(\widetilde M_i\) by
\[
        \frac{d}{dt}\phi_i(t)
        =
        -\nabla_{\widetilde g_i(t)}
        \widetilde f_i(t)\circ\phi_i(t),
        \qquad
        \phi_i(0)=\operatorname{Id},
\]
and define
\[
        \widehat{\widetilde g}_i(t)
        :=
        \phi_i(t)^*\widetilde g_i(t).
\]
By the terminal smooth convergence and the choice of the closed
completion,
$\widehat{\widetilde g}_i(t_i)\in\mathcal U,$
where \(\mathcal U\) is the Sobolev neighborhood of the hyperbolic orbit
appearing in Proposition~\ref{LSpath}.  Set
\[
        \mathcal I_i
        :=
        \left\{
        t\in[0,t_i]:
        \widehat{\widetilde g}_i(t)\in\mathcal U
        \right\},
        \qquad
        \mathcal O_i
        :=
        [0,t_i]\setminus\mathcal I_i .
\]
For every regular \(t\in\mathcal I_i\), Proposition~\ref{LSpath}
implies
\begin{equation}
\label{eq:LS-on-pulled-back-core}
        \widetilde e_i(t)^{1-\theta}
        \leq
        C_{\mathrm{LS}}
        \widetilde{\mathcal N}_i(t),
\end{equation}
where \(C_{\mathrm{LS}}<\infty\) and
\(\theta\in(0,\frac12]\) are independent of \(i\).

\noindent Fix \(\varepsilon_0>0\), and let
\[
        \mathscr K_{\varepsilon_0}
        :=
        \left\{
        \widehat{\widetilde g}_i(t):
        i\geq i_0,\quad
        t\in[0,t_i]\ \text{regular},\quad
        \widetilde e_i(t)\leq\varepsilon_0
        \right\}.
\]
By the uniform bounded-geometry estimates for the closed completed
good components and the corresponding Cheeger--Gromov compactness
theorem, after fixing the \(H^s\)-gauge used in
Proposition~\ref{LSpath}, the class
\(\mathscr K_{\varepsilon_0}\) is relatively compact in \(H^s\)
modulo \(\operatorname{Diff}^{s+1}(\widetilde M)\). Its volumes are
uniformly bounded above and below. Lemma
\ref{lem:uniform-soliton-defect-gap} therefore gives a constant
\(\gamma_{\mathcal U}>0\), independent of \(i\) and \(t\), such
that
\begin{equation}
\label{eq:uniform-gradient-gap-away-from-U}
        \widehat{\widetilde g}_i(t)\notin\mathcal U,
        \qquad
        \widetilde e_i(t)\leq\varepsilon_0
        \quad\Longrightarrow\quad
        \widetilde{\mathcal N}_i(t)
        \geq\gamma_{\mathcal U}.
\end{equation}
For all sufficiently large \(i\), one has
\(\varepsilon_i<\varepsilon_0\) by \eqref{eq:closed-entropy-budget}.
Hence \eqref{eq:uniform-gradient-gap-away-from-U} applies at every
regular time \(t\in\mathcal O_i\). From \eqref{eq:closed-entropy-budget},
\begin{equation}
\label{eq:outside-core-L1}
\begin{aligned}
        |\mathcal O_i|
        &\leq
        \frac{\varepsilon_i}
        {c_0\gamma_{\mathcal U}^{\,2}},\int_{\mathcal O_i}
        \widetilde{\mathcal N}_i(t)\,dt
        \leq
        \frac{\varepsilon_i}
        {c_0\gamma_{\mathcal U}}.
\end{aligned}
\end{equation}
On every regular connected component of \(\mathcal I_i\) on which
\(\widetilde e_i>0\), equations
\eqref{eq:closed-entropy-dissipation} and
\eqref{eq:LS-on-pulled-back-core} yield
\[
        -\frac{d}{dt}\widetilde e_i(t)^\theta
        \geq
        \frac{c_0\theta}{C_{\mathrm{LS}}}
        \widetilde{\mathcal N}_i(t).
\]
Summing over the regular components of \(\mathcal I_i\), and accounting
for the entropy jumps at surgery times, gives
\begin{equation}
\label{eq:inside-core-L1}
        \int_{\mathcal I_i}
        \widetilde{\mathcal N}_i(t)\,dt
        \leq
        \frac{C_{\mathrm{LS}}}{c_0\theta}
        \left(
        \widetilde e_i(0)^\theta+\sigma_i
        \right).
\end{equation}
Combining \eqref{eq:outside-core-L1} and
\eqref{eq:inside-core-L1}, we obtain
\begin{equation}
\label{eq:core-soliton-path}
        \int_0^{t_i}
        \widetilde{\mathcal N}_i(t)\,dt
        \leq
        \frac{C_{\mathrm{LS}}}{c_0\theta}
        \left(
        \widetilde e_i(0)^\theta+\sigma_i
        \right)
        +
        \frac{\varepsilon_i}
        {c_0\gamma_{\mathcal U}}
        \longrightarrow0.
\end{equation}
In particular, the bound is independent of \(t_i\). This is the main reason we need these additional estimates so that the relevant norm of the deformation tensor is uniform in time. 

\noindent The uniform bounded-geometry estimates on the buffered completed flow,
together with the ground-state equation
\[
        -4\Delta_{\widetilde g_i(t)}
        e^{-\widetilde f_i(t)/2}
        +
        R(\widetilde g_i(t))
        e^{-\widetilde f_i(t)/2}
        =
        \widetilde\lambda_i(t)
        e^{-\widetilde f_i(t)/2},
\]
and the elliptic Harnack and Schauder estimates, give
\begin{equation}
\label{eq:entropy-density-equivalence-tube}
        0<c_f
        \leq
        e^{-\widetilde f_i(t)}
        \leq
        C_f<\infty
        \qquad
        \text{on }G_i^-(t),\quad
        0\leq t\leq t_i,
\end{equation}
with constants independent of \(i\) and of the regular time \(t\).
Since the volumes of the completed components are uniformly bounded
above and below, it follows that
\begin{equation}
\label{eq:unweighted-soliton-bound}
        \|\widetilde S_i(t)\|
        _{L^2(G_i^-(t),\,\widetilde g_i(t))}
        \leq
        C\widetilde{\mathcal N}_i(t).
\end{equation}
We provide a detailed account of these estimates involving Harnack and Schauder estimates in the next lemma \ref{lem:uniform-perelman-ground-state}, proven separately. 

\noindent Moreover, since \(R(\widetilde g_i(t))\geq-6\),
\[
        0\leq\widetilde\lambda_i(t)+6
        \leq
        \frac{1}
        {\operatorname{Vol}_{\widetilde g_i(t)}(\widetilde M_i)}
        \int_{\widetilde M_i}
        \bigl(R(\widetilde g_i(t))+6\bigr)
        \,d\mu_{\widetilde g_i(t)}.
\]
The scalar-defect budget on the closed completed component therefore
implies
\begin{equation}
\label{eq:scalar-mode-time-integral}
        \int_0^{t_i}
        \bigl(\widetilde\lambda_i(t)+6\bigr)\,dt
        \longrightarrow0.
\end{equation}
Finally, define
\[
        \widetilde{\mathcal D}_i(t)
        :=
        \operatorname{Ric}_{\widetilde g_i(t)}
        +
        \nabla_{\widetilde g_i(t)}^2\widetilde f_i(t)
        +
        2\widetilde g_i(t).
\]
Then
\[
        \widetilde{\mathcal D}_i
        =
        \widetilde S_i
        +
        \frac{\widetilde\lambda_i+6}{3}
        \widetilde g_i.
\]
Since
the volumes of \(G_i^-(t)\) are uniformly bounded, equations
\eqref{eq:core-soliton-path},
\eqref{eq:unweighted-soliton-bound}, and
\eqref{eq:scalar-mode-time-integral} yield
\begin{equation}
\label{eq:modified-core-path-length}
\begin{aligned}
        L_i
        &:=
        \int_0^{t_i}
        \|\widetilde{\mathcal D}_i(t)\|
        _{L^2(G_i^-(t),\,\widetilde g_i(t))}
        \,dt                                                    \\
        &\leq
        C\int_0^{t_i}
        \widetilde{\mathcal N}_i(t)\,dt
        +
        C\int_0^{t_i}
        \bigl(\widetilde\lambda_i(t)+6\bigr)\,dt
        \longrightarrow0.
\end{aligned}
\end{equation}

\noindent
We now return to the cutoff flow \(\Psi_i(t)\) introduced above and
use the path-length estimate \eqref{eq:modified-core-path-length}.
For \(0\leq t\leq t_i\), set
\[
        \Theta_i(t)
        :=
        \Psi_i(t)\circ\Psi_i(t_i)^{-1}\circ\psi_i^t
        :
        K_i^+\longrightarrow G_i^+(t).
\]
Then
\[
        \Theta_i(t_i)=\psi_i^t,
        \qquad
        \widehat g_i(t):=\Theta_i(t)^*g_i(t),
\]
and the cutoff is identically one on the inner tube \(G_i^-(t)\), where
the deformation tensor is precisely
\(\widetilde{\mathcal D}_i(t)\). \noindent
Since the volumes of the completed components are uniformly bounded,
Hölder's inequality gives
\[
        \|\widetilde{\mathcal D}_i(t)\|
        _{L^1(G_i^-(t),g_i(t))}
        \leq
        C\|\widetilde{\mathcal D}_i(t)\|
        _{L^2(G_i^-(t),g_i(t))}.
\]
Thus the \(L_t^1L_x^2\)-estimate
\eqref{eq:modified-core-path-length} implies the corresponding
\(L_t^1L_x^1\)-estimate. We retain the \(L^2\)-formulation, which is
suited directly to the stopping-time argument below.

\noindent
We now pull the terminal comparison back to the initial slice. Set
\[
        \check g_i(t):=\Psi_i(t)^*g_i(t)
        \qquad\text{on }G_i^-(0),
\]
and define \(J_i(t,x)\) by
\[
        d\mu_{\check g_i(t)}(x)
        =
        J_i(t,x)\,d\mu_{g_i(0)}(x).
\]
Taking the trace of
\eqref{eq:modified-flow-metric-evolution}, we obtain
\[
        \frac{d}{dt}\log J_i(t,x)
        =
        -\operatorname{tr}_{g_i(t)}
        \widetilde{\mathcal D}_i(t)
        \bigl(\Psi_i(t,x)\bigr).
\]
Consequently,
\begin{equation}
\label{eq:modified-Jacobian-control}
        |\log J_i(t,x)|
        \leq
        \sqrt{3}
        \int_0^t
        |\widetilde{\mathcal D}_i(s)|_{g_i(s)}
        \bigl(\Psi_i(s,x)\bigr)\,ds .
\end{equation}

\noindent
Let
\[
        \eta_i:=L_i^{1/2}+i^{-1},
\]
and, for \(x\in G_i^-(0)\), define
\[
        \tau_i(x)
        :=
        \inf\left\{
        t\in[0,t_i]:
        \int_0^t
        |\widetilde{\mathcal D}_i(s)|_{g_i(s)}
        \bigl(\Psi_i(s,x)\bigr)\,ds
        \geq\eta_i
        \right\},
\]
with the convention \(\tau_i(x)=t_i\) when the threshold is not
attained. Set
\[
        \mathcal A_i^\sharp(x)
        :=
        \int_0^{\tau_i(x)}
        |\widetilde{\mathcal D}_i(s)|_{g_i(s)}
        \bigl(\Psi_i(s,x)\bigr)\,ds
\]
and
\[
        E_i(t):=
        \{x\in G_i^-(0):t\leq\tau_i(x)\}.
\]
For \(x\in E_i(t)\), equation
\eqref{eq:modified-Jacobian-control} gives
\begin{equation}
\label{eq:stopped-Jacobian-bound}
        e^{-\sqrt{3}\eta_i}
        \leq
        J_i(t,x)
        \leq
        e^{\sqrt{3}\eta_i}.
\end{equation}
By Minkowski's integral inequality, followed by
\eqref{eq:stopped-Jacobian-bound} and change of variables under
\(\Psi_i(t)\),
\begin{align}
        \|\mathcal A_i^\sharp\|
        _{L^2(G_i^-(0),g_i(0))}
        &\leq
        \int_0^{t_i}
        \left(
        \int_{E_i(t)}
        |\widetilde{\mathcal D}_i(t)|_{g_i(t)}^2
        \bigl(\Psi_i(t,x)\bigr)
        \,d\mu_{g_i(0)}(x)
        \right)^{1/2}dt
        \notag\\
        &\leq
        e^{\frac{\sqrt{3}}{2}\eta_i}
        \int_0^{t_i}
        \|\widetilde{\mathcal D}_i(t)\|
        _{L^2(G_i^-(t),g_i(t))}\,dt
        \notag\\
        &=
        e^{\frac{\sqrt{3}}{2}\eta_i}L_i .
\label{eq:stopped-accumulation-L2}
\end{align}

\noindent
Define
\[
        B_i^0
        :=
        \left\{
        x\in G_i^-(0):
        \int_0^{t_i}
        |\widetilde{\mathcal D}_i(s)|_{g_i(s)}
        \bigl(\Psi_i(s,x)\bigr)\,ds
        \geq\eta_i
        \right\}.
\]
For \(x\in B_i^0\), continuity in time gives
\(\mathcal A_i^\sharp(x)=\eta_i\). Hence
\eqref{eq:stopped-accumulation-L2} yields
\begin{equation}
\label{eq:modified-initial-bad-volume}
\begin{aligned}
        \operatorname{Vol}_{g_i(0)}(B_i^0)
        &\leq
        e^{\sqrt{3}\eta_i}
        \frac{L_i^2}{\eta_i^2}
        \longrightarrow0 .
\end{aligned}
\end{equation}
Here we used \(\eta_i\geq L_i^{1/2}\), and therefore
\(L_i^2/\eta_i^2\leq L_i\).

\noindent
Set
\[
        B_i^t:=\Psi_i(t_i)(B_i^0)\subset G_i^-(t_i).
\]
Equivalently, writing
\[
        \Psi_i^{\,t_i}(s,y)
        :=
        \Psi_i(s)\circ\Psi_i(t_i)^{-1}(y),
\]
one has
\[
        B_i^t
        =
        \left\{
        y\in G_i^-(t_i):
        \int_0^{t_i}
        |\widetilde{\mathcal D}_i(s)|_{g_i(s)}
        \bigl(\Psi_i^{\,t_i}(s,y)\bigr)\,ds
        \geq\eta_i
        \right\}.
\]
Applying the preceding stopped argument to the reversed family
\(s\mapsto\Psi_i^{\,t_i}(t_i-s,\cdot)\) gives
\begin{equation}
\label{eq:modified-terminal-bad-volume}
        \operatorname{Vol}_{g_i(t_i)}(B_i^t)
        \leq
        e^{\sqrt{3}\eta_i}
        \frac{L_i^2}{\eta_i^2}
        \longrightarrow0 .
\end{equation}
We now also record the volume of the initial traces discarded before the
stopped argument. Let \(\beta_i^t\to0\) denote the total
\(g_i(t_i)\)-volume of the terminal bad set, the buffer
\(G_i^t\setminus G_i^{t,-}\), and the additional survival-removal set.
Choose \(\kappa_i\downarrow0\) so that
\[
        \frac{\delta_i}{\kappa_i}\to0.
\]
The ordinary Ricci-flow Jacobian along surviving worldlines satisfies
\[
        J_i^{\mathrm{RF}}(t,x)
        =
        \exp\left(
        -\int_0^t
        \bigl(R(g_i(s))+6\bigr)(P_{0,s}(x))\,ds
        \right).
\]
Using the scalar-defect budget and the total surgery volume-loss
estimate, one obtains
\begin{equation}
\label{eq:initial-trace-volume}
\begin{aligned}
        \operatorname{Vol}_{g_i(0)}
        \bigl(M_i(0)\setminus G_i^-(0)\bigr)
        &\leq
        \frac{\delta_i}{1-e^{-\kappa_i}}
        +
        e^{\kappa_i}\bigl(\beta_i^t+\delta_i\bigr)
        \longrightarrow0.
\end{aligned}
\end{equation}
The first term accounts for worldlines with accumulated scalar defect
larger than \(\kappa_i\), while the second accounts for terminally
discarded regions and worldlines removed at surgery. Now set
\[
        G_i^0:=G_i^-(0)\setminus B_i^0,
        H_i^t:=\Psi_i(t_i)(G_i^0)
        =
        G_i^{t,-}\setminus B_i^t,
\]
and
\[
        K_i^0:=(\psi_i^t)^{-1}(H_i^t)\subset K_i^-.
\]
Define
\[
        \psi_i^0
        :=
        \Psi_i(t_i)^{-1}\circ
        \psi_i^t\big|_{K_i^0}
        :
        K_i^0\longrightarrow G_i^0.
\]
These maps are restrictions of smooth diffeomorphisms and hence are
diffeomorphisms onto their images. From
\eqref{eq:modified-initial-bad-volume} and
\eqref{eq:initial-trace-volume},
\[
        \operatorname{Vol}_{g_i(0)}
        (M_i(0)\setminus G_i^0)\to0.
\]
Moreover, the terminal \(C^0\)-comparison and
\eqref{eq:modified-terminal-bad-volume} imply
\[
        \operatorname{Vol}_h(K_i^-\setminus K_i^0)\to0.
\]
Since
\[
        \operatorname{Vol}_h(M\setminus K_i^-)\to0,
\]
it follows that
\[
        \operatorname{Vol}_h(M\setminus K_i^0)\to0.
\]
Finally, for \(x\in G_i^0\) and \(0\neq v\in T_xG_i^0\),
\eqref{eq:modified-flow-metric-evolution} gives
\[
        \left|
        \frac{d}{dt}
        \log\bigl(\check g_i(t)_x(v,v)\bigr)
        \right|
        \leq
        2|\widetilde{\mathcal D}_i(t)|_{g_i(t)}
        (\Psi_i(t,x)).
\]
By the definition of \(G_i^0\), integration over \([0,t_i]\) yields
\[
        e^{-2\eta_i}g_i(0)
        \leq
        \Psi_i(t_i)^*g_i(t_i)
        \leq
        e^{2\eta_i}g_i(0),
\]
or equivalently,
\[
        e^{-2\eta_i}\Psi_i(t_i)^*g_i(t_i)
        \leq
        g_i(0)
        \leq
        e^{2\eta_i}\Psi_i(t_i)^*g_i(t_i)
\]
as quadratic forms.
Pulling this inequality back by \(\psi_i^0\), and using
$\Psi_i(t_i)\circ\psi_i^0=\psi_i^t,$
we obtain
\[
        e^{-2\eta_i}(\psi_i^t)^*g_i(t_i)
        \leq
        (\psi_i^0)^*g_i(0)
        \leq
        e^{2\eta_i}(\psi_i^t)^*g_i(t_i).
\]
Combining this with
\[
        (1-\varepsilon_i^t)h
        \leq
        (\psi_i^t)^*g_i(t_i)
        \leq
        (1+\varepsilon_i^t)h
\]
gives
\[
        e^{-2\eta_i}(1-\varepsilon_i^t)h
        \leq
        (\psi_i^0)^*g_i(0)
        \leq
        e^{2\eta_i}(1+\varepsilon_i^t)h
\]
on \(K_i^0\). Since
\(\eta_i\to0\) and \(\varepsilon_i^t\to0\),
\[
        \|(\psi_i^0)^*g_i(0)-h
        \|_{C^0(K_i^0,h)}
        \longrightarrow0.
\]
This completes the proof.
\end{proof}
\noindent Now we state and prove the following lemma regarding limit space used in the proof of the proposition \ref{prop:flowback-c0-comparison-map}.
\begin{lemma}
\label{lem:uniform-perelman-ground-state}
Let
\[
        \bigl(\widetilde M_i,\widetilde g_i(t)\bigr),
        \qquad 0\leq t\leq t_i,
\]
be the closed completed flows appearing in Proposition~\ref{prop:flowback-c0-comparison-map}.
Assume that there exist constants
$K<\infty,~\iota>0,~0<v_0\leq V_0<\infty$,
independent of \(i\) and of the regular time \(t\), such that
\begin{equation}
\label{eq:uniform-completed-bounded-geometry}
        \bigl\|
                \operatorname{Rm}_{\widetilde g_i(t)}
        \bigr\|_{L^\infty(\widetilde M_i,\widetilde g_i(t))}
        \leq K,
        \qquad
        \operatorname{inj}_{\widetilde g_i(t)}(\widetilde M_i)
        \geq \iota,
\end{equation}
and
\begin{equation}
\label{eq:uniform-completed-volume-bounds}
        v_0
        \leq
        \operatorname{Vol}_{\widetilde g_i(t)}(\widetilde M_i)
        \leq
        V_0.
\end{equation}
Let \(\widetilde f_i(t)\) be the normalized minimizer of
Perelman's \(\lambda\)-functional,
\[
        \int_{\widetilde M_i}
        e^{-\widetilde f_i(t)}
        \,d\mu_{\widetilde g_i(t)}
        =1,
\]
and write
\[
        \widetilde\lambda_i(t)
        :=
        \lambda[\widetilde g_i(t)].
\]
Then there exist constants
$ 0<c_f\leq C_f<\infty$,
depending only on \(K,\iota,v_0\), and \(V_0\), such that, for
every \(i\) and every regular \(t\in[0,t_i]\),
\begin{equation}
\label{eq:uniform-perelman-density-bound}
        c_f
        \leq
        e^{-\widetilde f_i(t)}
        \leq
        C_f
        \qquad\text{on }\widetilde M_i.
\end{equation}
In particular, the same estimate holds on \(G_i^-(t)\).
Moreover, if
\[
        \widetilde S_i(t)
        :=
        \operatorname{Ric}_{\widetilde g_i(t)}
        +
        \nabla_{\widetilde g_i(t)}^2\widetilde f_i(t)
        -
        \frac{\widetilde\lambda_i(t)}{3}
        \widetilde g_i(t)
\]
and
\[
        \widetilde{\mathcal N}_i(t)
        :=
        \operatorname{Vol}_{\widetilde g_i(t)}
        (\widetilde M_i)^{1/3}
        \left(
        \int_{\widetilde M_i}
        |\widetilde S_i(t)|_{\widetilde g_i(t)}^2
        e^{-\widetilde f_i(t)}
        \,d\mu_{\widetilde g_i(t)}
        \right)^{1/2},
\]
then
\begin{equation}
\label{eq:weighted-to-unweighted-soliton-defect}
        \|\widetilde S_i(t)\|_
        {L^2(G_i^-(t),\widetilde g_i(t))}
        \leq
        C\,\widetilde{\mathcal N}_i(t),
\end{equation}
where \(C\) is independent of \(i\) and \(t\).
\end{lemma}

\begin{proof}
Set
\[
        \widetilde u_i(t)
        :=
        e^{-\widetilde f_i(t)/2}.
\]
Then \(\widetilde u_i(t)>0\),
\[
        \int_{\widetilde M_i}
        \widetilde u_i(t)^2
        \,d\mu_{\widetilde g_i(t)}
        =1,
\]
and \(\widetilde u_i(t)\) satisfies the ground-state equation
\begin{equation}
\label{eq:perelman-ground-state-equation}
        -4\Delta_{\widetilde g_i(t)}
        \widetilde u_i(t)
        +
        R(\widetilde g_i(t))\widetilde u_i(t)
        =
        \widetilde\lambda_i(t)\widetilde u_i(t).
\end{equation}
The curvature bound in
\eqref{eq:uniform-completed-bounded-geometry} gives
\[
        \|R(\widetilde g_i(t))\|_{L^\infty}
        \leq C(K).
\]
By the variational characterization of
\(\widetilde\lambda_i(t)\),
\[
        \inf_{\widetilde M_i}R(\widetilde g_i(t))
        \leq
        \widetilde\lambda_i(t)
        \leq
        \frac{1}{
        \operatorname{Vol}_{\widetilde g_i(t)}(\widetilde M_i)}
        \int_{\widetilde M_i}
        R(\widetilde g_i(t))
        \,d\mu_{\widetilde g_i(t)}.
\]
Consequently,
\[
        |\widetilde\lambda_i(t)|
        \leq C(K),
\]
and hence
\[
        \Delta_{\widetilde g_i(t)}\widetilde u_i(t)
        =
        q_i(t)\widetilde u_i(t),
        \qquad
        \|q_i(t)\|_{L^\infty}
        \leq C(K).
\]

\noindent The curvature and injectivity-radius bounds provide a radius
\(r_0=r_0(K,\iota)>0\) on which the local elliptic Harnack
inequality is uniform. Thus, for every
\(x\in\widetilde M_i\),
\[
        \sup_{B_{\widetilde g_i(t)}(x,r_0)}
        \widetilde u_i(t)
        \leq
        C_H
        \inf_{B_{\widetilde g_i(t)}(x,r_0)}
        \widetilde u_i(t),
\]
where \(C_H\) is independent of \(i,t\), and \(x\).

\noindent The same geometric bounds give a uniform positive lower bound
for the volume of every ball of radius \(r_0/4\). Together with
the upper volume bound in
\eqref{eq:uniform-completed-volume-bounds}, this implies that
\(\widetilde M_i\) admits a cover by at most \(N_0\) balls of
radius \(r_0\), where \(N_0\) is independent of \(i\) and \(t\).
Since \(\widetilde M_i\) is connected, iteration of the local
Harnack inequality along an overlapping chain of these balls
gives
\begin{equation}
\label{eq:global-ground-state-harnack}
        \sup_{\widetilde M_i}\widetilde u_i(t)
        \leq
        H
        \inf_{\widetilde M_i}\widetilde u_i(t),
\end{equation}
for a uniform constant \(H<\infty\) independent of $i$.

\noindent Using the normalization and
\eqref{eq:uniform-completed-volume-bounds}, we obtain
\[
\begin{split}
        1
        &\leq
        V_0
        \left(
        \sup_{\widetilde M_i}\widetilde u_i(t)
        \right)^2
        \leq
        V_0H^2
        \left(
        \inf_{\widetilde M_i}\widetilde u_i(t)
        \right)^2,
\end{split}
\]
and hence
\[
        \inf_{\widetilde M_i}\widetilde u_i(t)
        \geq
        \frac{1}{H\sqrt{V_0}}.
\]
Similarly,
\[
\begin{split}
        1
        &\geq
        v_0
        \left(
        \inf_{\widetilde M_i}\widetilde u_i(t)
        \right)^2
        \geq
        \frac{v_0}{H^2}
        \left(
        \sup_{\widetilde M_i}\widetilde u_i(t)
        \right)^2,
\end{split}
\]
so that
\[
        \sup_{\widetilde M_i}\widetilde u_i(t)
        \leq
        \frac{H}{\sqrt{v_0}}.
\]
Squaring these inequalities proves
\eqref{eq:uniform-perelman-density-bound}. Finally, using the lower bound in
\eqref{eq:uniform-perelman-density-bound}, we have
\[
\begin{split}
        \|\widetilde S_i(t)\|_
        {L^2(G_i^-(t),\widetilde g_i(t))}^2
        &\leq
        c_f^{-1}
        \int_{G_i^-(t)}
        |\widetilde S_i(t)|_{\widetilde g_i(t)}^2
        e^{-\widetilde f_i(t)}
        \,d\mu_{\widetilde g_i(t)}
        \\
        &\leq
        c_f^{-1}
        \operatorname{Vol}_{\widetilde g_i(t)}
        (\widetilde M_i)^{-2/3}
        \widetilde{\mathcal N}_i(t)^2
        \\
        &\leq
        c_f^{-1}v_0^{-2/3}
        \widetilde{\mathcal N}_i(t)^2.
\end{split}
\]
This proves
\eqref{eq:weighted-to-unweighted-soliton-defect}.
\end{proof}

\subsection{Proof of the main theorems}
\label{subsec:proof-main-theorems}

In this section we combine the propositions \ref{prop:entropy-good-time-c0-hyperbolic-borel}-\ref{prop:flowback-c0-comparison-map} proved in the previous two sections and provide the proof of the main theorems. We first prove the volume-stability theorem.  We state the argument in the
form needed below, namely as a tensorial convergence statement on large
good sets.  
\begin{proof}[Proof of Theorem \ref{main1}]
We give a complete proof and repeat some of the previous definitions for the sake of self-contained proof. First let
$(M,h)$
be the fixed closed hyperbolic three-manifold, normalized by
\[
        \operatorname{Ric}_{h}=-2h,
        \qquad
        R(h)=-6.
\]
Let \(g_i(0)\) be a sequence of smooth metrics on the same smooth manifold
\(M\) satisfying
\[
        R(g_i(0))\geq -6,
        \qquad
        \operatorname{Vol}_{g_i(0)}(M)
        =
        \operatorname{Vol}_{h}(M)+\delta_i,
        \qquad
        \delta_i\to0.
\]
Let \(g_i(t)\), \(t\in[0,T]\), be the corresponding normalized Ricci flows
\[
        \partial_t g_i
        =
        -2\bigl(\operatorname{Ric}_{g_i}+2g_i\bigr),
        \qquad
        g_i(0)=g_i.
\]
Set
\[
        Q_i(t):=R(g_i(t))+6,
        \qquad
        E_i(t):=\operatorname{Ric}_{g_i(t)}+2g_i(t).
\]
Then
\[
        (\partial_t-\Delta_{g_i}+4)Q_i
        =
        2|E_i|_{g_i}^{2}.
\]
Hence \(Q_i(0)\geq0\) implies
\[
        Q_i(t)\geq0
\]
for every \(t\in[0,T]\). The volume satisfies
\[
        \frac{d}{dt}\operatorname{Vol}_{g_i(t)}(M)
        =
        -\int_M Q_i(t)\,d\mu_{g_i(t)}.
\]
Since \(Q_i(t)\geq0\), the volume is non-increasing.  Moreover, since the
underlying smooth manifold is the fixed hyperbolic topological manifold
\(M\), the hyperbolic volume comparison theorem gives
\[
        \operatorname{Vol}_{g_i(t)}(M)
        \geq
        \operatorname{Vol}_{h}(M)
\]
for all \(t\in[0,T]\).  Therefore
\[
        \operatorname{Vol}_{h}(M)
        \leq
        \operatorname{Vol}_{g_i(t)}(M)
        \leq
        \operatorname{Vol}_{g_i(0)}(M)
        =
        \operatorname{Vol}_{h}(M)+\delta_i.
\]
Consequently,
\[
        \operatorname{Vol}_{g_i(t)}(M)\to\operatorname{Vol}_{h}(M)
\]
uniformly on \([0,T]\), and
\[
        \int_0^T\int_M Q_i(t)\,d\mu_{g_i(t)}\,dt
        =
        \operatorname{Vol}_{g_i(0)}(M)
        -
        \operatorname{Vol}_{g_i(T)}(M)
        \leq
        \delta_i.
\]
\noindent Let \(f_i(t)\) be the minimizer for Perelman's \(\lambda\)-functional,
normalized by
\[
        \int_M e^{-f_i(t)}\,d\mu_{g_i(t)}=1,
\]
and define
\[
        \mathcal S_i(t)
        :=
        \operatorname{Ric}_{g_i(t)}
        +
        \nabla^2 f_i(t)
        -
        \frac{\lambda[g_i(t)]}{3}g_i(t).
\]
By the entropy-deficit estimate and Perelman's monotonicity formula as in the proposition \ref{monotonicity},
\[
        \int_0^T
        \operatorname{Vol}_{g_i(t)}(M)^{2/3}
        \int_M
        |\mathcal S_i(t)|_{g_i(t)}^2
        e^{-f_i(t)}\,d\mu_{g_i(t)}\,dt
        \to0.
\]
Hence using propositions \ref{prop:entropy-good-time-c0-hyperbolic-borel} one may choose
$t_i\in[T/2,T]$ with $T\to\infty$
such that, with
\[
        \widehat g_i:=g_i(t_i),
        \qquad
        \widehat f_i:=f_i(t_i),
        \qquad
        \widehat\lambda_i:=\lambda[\widehat g_i],
\]
one has
\[
        \int_M
        |\widehat{\mathcal S}_i|_{\widehat g_i}^2
        e^{-\widehat f_i}\,d\mu_{\widehat g_i}
        \to0,
\]
\[
        \operatorname{Vol}_{\widehat g_i}(M)\to\operatorname{Vol}_{h}(M),
        \qquad
        \widehat\lambda_i\to -6.
\]
\noindent By Proposition \ref{prop:entropy-good-time-c0-hyperbolic-borel}, after
passing to a subsequence there exist 
$\widehat Z_i\subset M,  
        \widehat G_i:=M\setminus \widehat Z_i,
$
and maps
\[
        \widehat\Phi_i:\widehat G_i\longrightarrow M
\]
such that
$\operatorname{Vol}_{\widehat g_i}(\widehat Z_i)\to0,$
and
\[
        \sup_{y\in \widehat G_i}
        \left|
        \widehat g_i-\widehat\Phi_i^*h
        \right|_{\widehat\Phi_i^*h}(y)
        \to0.
\]
Equivalently, for some \(\varepsilon_i\to0\),
\[
        (1-\varepsilon_i)\widehat\Phi_i^*h
        \leq
        \widehat g_i
        \leq
        (1+\varepsilon_i)\widehat\Phi_i^*h
\]
as quadratic forms on \(\widehat G_i\). It remains to pull this information back to the initial slice.  Let
\[
        \Psi_i(t):G_i^0\longrightarrow G_i(t)
\]
be the spacetime tube map from Proposition
\ref{prop:flowback-c0-comparison-map}.  Thus \(G_i(t_i)\subset\widehat G_i\),
and on \(G_i^0\) one has due to the result of proposition \ref{prop:flowback-c0-comparison-map}
\[
        e^{-2\eta_i}g_i(0)
        \leq
        \Psi_i(t_i)^*\widehat g_i
        \leq
        e^{2\eta_i}g_i(0),
        \qquad
        \eta_i\to0.
\]
Moreover, if $Z_i^0:=M\setminus G_i^0$,
then
\[
        \operatorname{Vol}_{g_i(0)}(Z_i^0)\to0.
\]
Define
$\Phi_i^0:G_i^0\longrightarrow M$ by
\[
        \Phi_i^0
        :=
        \widehat\Phi_i\circ\Psi_i(t_i).
\]
Combining the two quadratic-form estimates gives
\[
        e^{-2\eta_i}(1-\varepsilon_i)(\Phi_i^0)^*h
        \leq
        g_i(0)
        \leq
        e^{2\eta_i}(1+\varepsilon_i)(\Phi_i^0)^*h
\]
on \(G_i^0\).  Hence
\[
        \sup_{x\in G_i^0}
        \left|
        g_i(0)-(\Phi_i^0)^*h
        \right|_{(\Phi_i^0)^*h}(x)
        \to0.
\]
This proves the claimed \(C^0\)-convergence on the good part.
\end{proof}

\begin{remark}
In addition, the no short-circuit (i.e., the distance between the two points on $(M,g_{i})$ is not realized by a rectifiable curve lying the bad set) condition is imposed on \(G_i^0\), then
the intrinsic length distances on \(G_i^0\) are distorted by \(o(1)\), and
the preceding \(C^0\)-comparison implies the measured Gromov--Hausdorff
statement in Theorem \ref{main1}.  Without this additional no short-circuit
assumption, the conclusion is only the tensorial \(C^0\)-convergence above.    
\end{remark}

\begin{lemma}
\label{lem:cmc-normalization}
Let \((M,g,k)\) be a smooth vacuum CMC initial data set on a closed
three-manifold.  Write
\[
        \tau:=\operatorname{tr}_g k,
        \qquad
        \sigma:=k-\frac{\tau}{3}g.
\]
Assume
\[
        \tau<0,
        \qquad
        \nabla\tau=0.
\]
Then
\[
        \operatorname{tr}_g\sigma=0,
        \qquad
        \operatorname{div}_g\sigma=0,
\]
and the Hamiltonian constraint is
\[
        R(g)+\frac{2}{3}\tau^2
        =
        |\sigma|_g^2.
\]
Define
\[
        \bar g:=\frac{\tau^2}{9}g,
        \qquad
        \bar\Sigma:=\frac{|\tau|}{3}\sigma.
\]
Then
\[
        R(\bar g)+6
        =
        |\bar\Sigma|_{\bar g}^2,
        \qquad
        \operatorname{div}_{\bar g}\bar\Sigma=0,
        \qquad
        \operatorname{tr}_{\bar g}\bar\Sigma=0.
\]
Moreover,
\[
        \operatorname{Vol}_{\bar g}(M)
        =
        \frac{(-\tau)^3}{27}\operatorname{Vol}_{g}(M).
\]
\end{lemma}

\begin{proof} The proof is a straightforward scaling argument.
The vacuum constraints are
\[
        R(g)-|k|_g^2+\tau^2=0,
\]
and
\[
        \nabla^j k_{ij}-\nabla_i\tau=0.
\]
Since \(\nabla\tau=0\), the momentum constraint gives
\[
        \nabla^j\sigma_{ij}=0.
\]
Also,
\[
        \operatorname{tr}_g\sigma
        =
        \operatorname{tr}_g k-\tau=0.
\]
Since
\[
        k=\sigma+\frac{\tau}{3}g,
\]
we have
\[
        |k|_g^2
        =
        |\sigma|_g^2+\frac{\tau^2}{3}.
\]
Substituting this into the Hamiltonian constraint gives
\[
        R(g)+\frac{2}{3}\tau^2=|\sigma|_g^2.
\]
Now consider the following re-scaling
\[
        \alpha:=\frac{\tau^2}{9},
        \qquad
        \bar g=\alpha g.
\]
Since \(\alpha\) is constant,
\[
        R(\bar g)=\alpha^{-1}R(g).
\]
Moreover, for a covariant two-tensor \(T\),
\[
        |T|_{\bar g}^2=\alpha^{-2}|T|_g^2.
\]
With
\[
        \bar\Sigma=\alpha^{1/2}\sigma
        =
        \frac{|\tau|}{3}\sigma,
\]
we obtain
\[
        |\bar\Sigma|_{\bar g}^2
        =
        \alpha^{-1}|\sigma|_g^2.
\]
Therefore
\[
\begin{aligned}
        R(\bar g)+6
        &=
        \alpha^{-1}R(g)+6        =
        \alpha^{-1}
        \left(
        R(g)+6\alpha
        \right)=                 
        \alpha^{-1}
        \left(
        R(g)+\frac{2}{3}\tau^2
        \right)                 =
        \alpha^{-1}|\sigma|_g^2
        =
        |\bar\Sigma|_{\bar g}^2.
\end{aligned}
\]
The trace-free and divergence-free conditions are preserved under this
constant rescaling.  Finally, in dimension three,
\[
        d\mu_{\bar g}
        =
        \alpha^{3/2}d\mu_g
        =
        \frac{(-\tau)^3}{27}d\mu_g,
\]
because \(\tau<0\).  Hence
\[
        \operatorname{Vol}_{\bar g}(M)
        =
        \frac{(-\tau)^3}{27}\operatorname{Vol}_{g}(M).
\]
\end{proof}

\begin{lemma}
\label{lem:no-maximal-slice}
Let \((M,g,k)\) be a vacuum initial data set on a closed three-manifold of
negative Yamabe type.  If \(M\) is CMC, then
\[
        \tau:=\operatorname{tr}_g k
\]
cannot vanish identically.  In particular, an expanding CMC branch with
\(\tau<0\) cannot cross a maximal slice.
\end{lemma}

\begin{proof}
Suppose, for contradiction, that \(\tau=0\).  The Hamiltonian constraint
then gives
\[
        R(g)=|k|_g^2\geq0.
\]
Thus \(M\) admits a smooth metric of nonnegative scalar curvature.  This
contradicts the assumption that \(M\) is of negative Yamabe type.  Therefore
\(\tau\neq0\).  Since \(\tau\) is constant on each CMC slice, an expanding
branch with \(\tau<0\) cannot cross \(\tau=0\).
\end{proof}

\begin{proof}[Proof of Theorem \ref{main2}]
Let $(M,g_i,k_i)$ be a sequence of smooth vacuum CMC initial data sets on a fixed closed
hyperbolic three-manifold \(M\).  Write
\[
        \tau_i:=\operatorname{tr}_{g_i}k_i<0,
        \qquad
        \sigma_i:=k_i-\frac{\tau_i}{3}g_i.
\]
Since the data are CMC, \(\tau_i\) is constant on \(M\).  The vacuum
constraints give
\[
        R(g_i)-|k_i|_{g_i}^2+\tau_i^2=0,
        \qquad
        \operatorname{div}_{g_i}\sigma_i=0.
\]
By Lemma \ref{lem:cmc-normalization}, define
\[
        \bar g_i:=\frac{\tau_i^2}{9}g_i,
        \qquad
        \bar\Sigma_i:=\frac{|\tau_i|}{3}\sigma_i.
\]
Then
\[
        R(\bar g_i)+6
        =
        |\bar\Sigma_i|_{\bar g_i}^2\geq0,
\]
and
\[
        \operatorname{div}_{\bar g_i}\bar\Sigma_i=0,
        \qquad
        \operatorname{tr}_{\bar g_i}\bar\Sigma_i=0.
\]
In particular,
$R(\bar g_i)\geq -6$.
With the convention
\[
        \mathcal H(g_i,k_i)
        :=
        (-\tau_i)^3\operatorname{Vol}_{g_i}(M),
\]
Lemma \ref{lem:cmc-normalization} gives
\[
        \mathcal H(g_i,k_i)
        =
        27\operatorname{Vol}_{\bar g_i}(M).
\]
For the Lorentz-cone data over \((M,h)\), the normalized spatial metric is
\(h\) and the transverse-traceless part vanishes.  Hence the Lorentz-cone
Hamiltonian value is
\[
        \mathcal H_{\mathrm{LC}}
        =
        27\operatorname{Vol}_h(M).
\]
The assumption
\[
        \mathcal H(g_i,k_i)\to \mathcal H_{\mathrm{LC}}
\]
is therefore equivalent to
$\operatorname{Vol}_{\bar g_i}(M)\to \operatorname{Vol}_h(M)$.

\noindent Thus the sequence \((M,\bar g_i)\) satisfies precisely the hypotheses of
Theorem \ref{main1}.  Applying Theorem \ref{main1} to \(\bar g_i\), we
obtain 
\[
        Z_i\subset M,
        G_i:=M\setminus Z_i,
\]
and maps
\[
        \Phi_i:G_i\longrightarrow M
\]
such that
$\operatorname{Vol}_{\bar g_i}(Z_i)\to0, 
$
and
\[
        \sup_{x\in G_i}
        \left|
        \bar g_i-\Phi_i^*h
        \right|_{\Phi_i^*h}(x)
        \to0.
\]
Equivalently, for every \(\varepsilon>0\), for all sufficiently large \(i\),
\[
        (1-\varepsilon)\Phi_i^*h
        \leq
        \bar g_i
        \leq
        (1+\varepsilon)\Phi_i^*h
\]
on \(G_i\).
This proves the reduced-Hamiltonian stability theorem.
\end{proof}

\end{document}